\documentclass[11pt,a4paper]{article}
\usepackage[utf8]{inputenc}

\usepackage{lmodern}
\usepackage[T1]{fontenc}
\usepackage[english]{babel}
\usepackage{ifpdf}
\usepackage[left=1in, right=1in, top=1in, bottom=1in]{geometry}
\usepackage[dvipsnames]{xcolor}
\usepackage[colorlinks=true,linkcolor=blue,citecolor=red]{hyperref}
\usepackage[square,numbers]{natbib}
\usepackage{subcaption}      
\usepackage[labelfont=bf]{caption}
\hypersetup{
	pdftitle={},
	pdfauthor={}
} 

\usepackage{graphicx, amsmath, amsthm, amssymb, enumitem, mathrsfs} 
\usepackage{booktabs} 

\usepackage[capitalize,nameinlink]{cleveref}
\usepackage{rotating}

\Crefname{figure}{Figure}{Figures}

\usepackage{soul}
\usepackage[linesnumbered,ruled]{algorithm2e}

\usepackage{tikz}
\usepackage{tikz-3dplot} 
\usepackage{pgfplots, pgfplotstable}
\usetikzlibrary{3d} 
\usetikzlibrary{arrows.meta}

\newcommand{\supp}{{\mathsf{supp }\hspace{.075cm}}}

\newcommand{\simple}[1]{\mathrm{\mathsf{simple}}({#1})}

\newtheorem{theorem}{Theorem}[section]

\newtheorem{remark}[theorem]{Remark}
\newtheorem{definition}[theorem]{Definition}
\newtheorem{lemma}[theorem]{Lemma}

\newtheorem{example}[theorem]{Example}

\usepackage{xcolor, soul}

\title{\vspace{-1.5cm}Stochastically Evolving Graphs\\via Edit Semigroups}
\author{Fan Chung \& Sawyer Jack Robertson}
\date{\today}

\begin{document}
\maketitle

\begin{abstract}
    We investigate a randomly evolving process of subgraphs in an underlying host graph using the spectral theory of semigroups related to the Tsetlin library and hyperplane arrangements. Starting with some initial subgraph, at each iteration, we apply a randomly selected edit to the current subgraph. Such edits vary in nature from \emph{simple} edits consisting of adding or deleting an edge, or \emph{compound} edits which can affect several edges at once. This evolving process generates a random walk on the set of all possible subgraphs of the host graph. We show that the eigenvalues of this random walk can be naturally indexed by subsets of edges of the host graph. We also provide, in the case of simple edits, a closed-form formula for the eigenvectors of the transition probability matrix and a sharp bound for the rate of convergence of this random walk. We consider extensions to the case of compound edits; examples of this model include the previously studied Moran forest model and a dynamic random intersection graph model. Evolving graphs arise in a variety of fields ranging from deep learning and graph neural networks to epidemic modeling and social networks. Our random evolving process serves as a general stochastic model for sampling random subgraphs from a given graph.
    
    \vspace{.25cm}
    {\footnotesize

    }
\end{abstract}

\section{Introduction}\label{sec:introduction}

Many problems arising in dealing with large data and deep learning involve graphs that are dynamically evolving, with edges appearing and disappearing over time. Evolving graphs can be regarded as stochastic processes that have been extensively studied in the literature with a wide range of applications (see, e.g.,~\cite{anari2019log,cooper2007sampling}). In this paper we investigate a general edit-based model of stochastically evolving \emph{subgraphs} of a specified host graph. At each step of the process, an edit is randomly chosen and applied to the current set of edges.

We consider two types of edits. A \emph{simple edit} consists of adding or deleting a chosen edge, while a \emph{compound edit} involves adding or deleting multiple edges at once. Our graph evolving process satisfies the following memoryless property: namely, that each chosen edit is carried out regardless of the previous edits or existing graph. For example, a simple edit such as ``add edge $e$'' is applied regardless of whether $e$ belongs to the current subgraph, and has no effect if $e$ is already present. By treating edits and their products as elements of a semigroup, we will demonstrate that the edit semigroup forms a so-called \emph{left regular band}. Namely, for any two elements $x, y$ belonging to the edit semigroup, we have $x^2=x$ and the following memorylessness identity holds:
    \begin{align}\label{eq:memorylessness}
        xyx = xy.
    \end{align}
We then apply methods from semigroup spectral theory to the evolving graph process and derive closed-form expressions for the eigenvalues, stationary distribution, and eigenvectors of the transition probability matrix, and sharp mixing time bounds.

We remark that different choices of edits give rise to distinct graph evolving processes, which in turn have their own unique stationary distributions. In particular, we show that the stationary distribution of a certain graph edit Markov chain obtained from simple edits is equal in distribution to a corresponding edge-independent random graph model, including such examples as Erd\H{o}s-R\'{e}nyi graphs, power law graphs, and stochastic block models (see, e.g.,~\cref{ex:stationary-er} to~\cref{ex:stationary-cl}). When the Markov chain is obtained from compound edits, the stationary distribution is more complex. We consider two examples of our Markov chain obtained from applying compound edits: the first version has a stationary distribution known as the \emph{Moran forest model}, which is a random forest model that has applications in computational biology~\cite{bienvenuMoranForest2021}; and the second has a stationary distribution which corresponds to the random intersection graph model investigated by Godehardt and Jaworski~\cite{godehardt2001two}. 

Our graph evolving sequence can be regarded as a random walk on an associated directed state graph. Each node of the state graph corresponds to a subgraph of the host graph and the links of the state graph are defined and weighted by the graph evolving process of adding and deleting edges in the host graph. Two examples of state graphs formed by applying simple and compound edits are illustrated in~\cref{fig:state-graph-5} and~\cref{fig:state-graph-moran}, respectively. In spite of the enormous size of the state graph (to $2^m$ nodes where $m$ denotes the number of edges in the host graph), our methods facilitate a detailed investigation of the spectral properties of the state graph and the corresponding transition probability matrix.


Spectral analysis of random walks on semigroups dates back to Bidigare, Hanlon, and Rockmore~\cite{bidigare1999combinatorial}, who derived a closed-form expression for the eigenvalues of the transition probability matrix of a Markov chain based on book shuffling (Tsetlin library). Brown and Diaconis later obtained convergence bounds for random walks on hyperplane arrangements~\cite{brown1998random}. Brown gave a detailed treatment by exploring connections between several combinatorial structures and the eigenvalues of associated random walks~\cite{brownSemigroupsRingsMarkov2000}. Saliola later analyzed eigenvectors of transition probability matrices for these walks~\cite{saliola2012eigenvectors}. The first author and Graham used these methods for edge flipping games and voter models~\cite{chungEdgeFlippingGraphs2012} (see also the further work~\cite{butlerEdgeFlippingComplete2015,demirci2023mixing,chung2012hypergraph}).

Stochastically evolving graph models have been studied in a variety of settings; examples include edge-switching Markov chains for sampling degree-regular graphs~\cite{cooper2007sampling}, up-down walks on simplicial complexes for sampling spanning trees and matroid bases~\cite{anari2019log}, or edge-Markovian dynamic graphs~\cite{clementi2008flooding, akrida2020fast}. Various temporal graph models have appeared in a variety of applied settings, including large person-to-person communication networks~\cite{sekaraFundamentalStructuresDynamic2016}, interaction networks of molecular systems in biology~\cite{przytyckaDynamicInteractomeIts2010}, disease and epidemic models~\cite{pareEpidemicProcessesTimeVarying2018}, and phylogenetic trees~\cite{bienvenuMoranForest2021, durrettProbabilityModelsDNA2002}, among others~\cite{michailIntroductionTemporalGraphs2016, aggarwalEvolutionaryNetworkAnalysis2014, holmeTemporalNetworks2012}. Moreover, graph neural network models have been extensively studied to predict future behavior in these systems~\cite{wu2020evonet}, with applications in areas such as automatic robot design~\cite{wang2018neural}, time series forecasting~\cite{spadon2021pay}, and object detection~\cite{xu2022novel}, among others~\cite{skardingFoundationsModelingDynamic2021a}.

\begin{figure}[t!]
    \centering
    \includegraphics[width=0.45\linewidth]{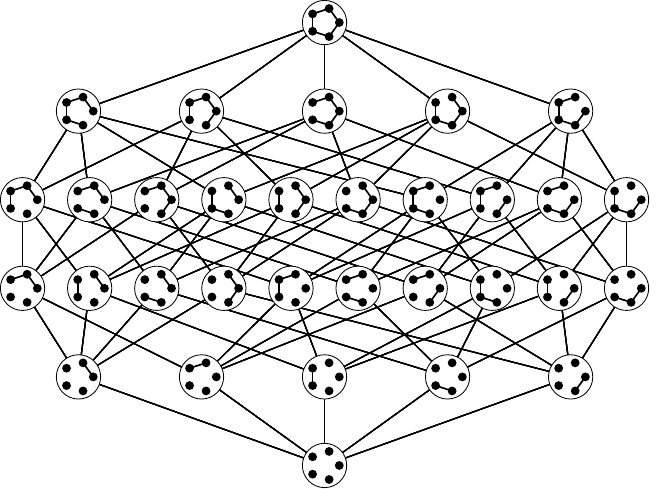}
    \caption{This state graph is associated with the simple edit process with the host graph corresponding to a cycle on $5$ edges. Edges are directed with possibly different weights; self-loops are not shown.}\label{fig:state-graph-5}
\end{figure}
\begin{figure}[t!]
    \centering
    \includegraphics[width=.45\linewidth]{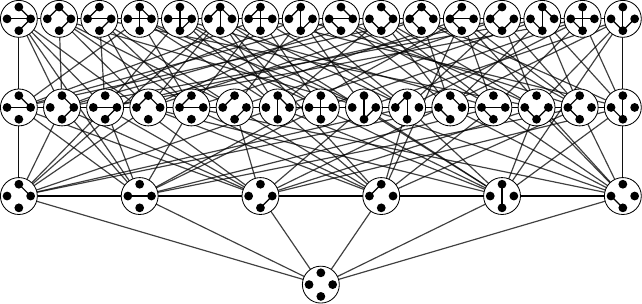}
    \caption{This state graph is associated with the Moran forest process with the host graph $\mathcal{H} = K_4$. Edges are directed with possibly different weights; self-loops are not shown.}\label{fig:state-graph-moran}
\end{figure}

Before we state several main results, we need some definitions. For a host graph $\mathcal{H}=(\mathcal{V},\mathcal{E})$ with a finite vertex set $\mathcal{V}$ and edge set $\mathcal{E}\subseteq \binom{\mathcal{V}}{2}$, a \emph{graph edit} is an idempotent map $x:2^{\mathcal{E}}\to 2^{\mathcal{E}}$ (i.e., $x^2=x$). A \emph{simple edit} adds or deletes a single edge.


\begin{definition}[Simple edit process]\label{defn:simple-edit-process}
    Given $0<p_e<1$ for each $e\in\mathcal{E}$, the \emph{simple edit process with edge probabilities $(p_e)_{e\in\mathcal{E}}$} is defined as follows. Starting from $G_0=(\mathcal{V},E_0)$, for each $t\ge 1$ choose $e\in\mathcal{E}$ uniformly at random and set 
        \begin{align*}
            E_t=
                \begin{cases}
                E_{t-1}\cup\{e\}, & \text{with probability } p_e,\\
                E_{t-1}\setminus\{e\}, & \text{with probability } 1-p_e.
                \end{cases}
        \end{align*}
    with $G_t = (\mathcal{V}, E_t)$ for $t\geq 0$.
\end{definition}

A simulation of a simple edit process appears in~\cref{fig:diffusion-on-graph}.

\begin{theorem}[Eigenvalues for the simple edit process]\label{thm:intro-eigenvalues}
    Consider the simple edit process $(G_t)_{t\geq 0}$ defined on a host graph $\mathcal{H}=(\mathcal{V}, \mathcal{E})$ with edge probabilities $(p_e)_{e\in\mathcal{E}}$. Then the transition probability matrix $\mathsf{P}$ for the associated Markov chain is diagonalizable and has eigenvalues $\lambda_T$ indexed by subsets $T\subseteq \mathcal{E}$, each occurring with multiplicity one, and which have the following form:
        \begin{align*}
            \lambda_T = \frac{|T|}{|\mathcal{E}|}
        \end{align*}
    where $|\mathcal{E}|$ is the cardinality of $\mathcal{E}$, and similarly for $T$. Equivalently, $\mathsf{P}$ has an eigenvalue $\frac{k}{|\mathcal{E}|}$ for each $0\leq k\leq |\mathcal{E}|$ with multiplicity ${|\mathcal{E}|\choose k}$. 
\end{theorem}

\begin{figure}[b!]
    \centering
    \includegraphics[width=0.75\linewidth]{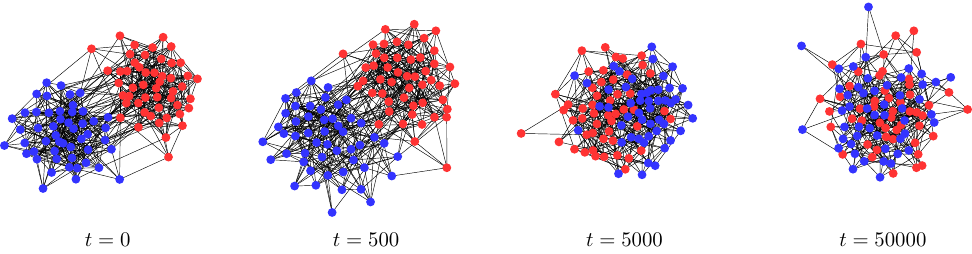}
    \caption{A simple graph edit process simulation on $n=100$ nodes and host graph $\mathcal{H} = K_{100}$ with edge probabilities $p_e = 0.075$ for each $e\in \mathcal{E}$. The initial state originates from a planted bisection model with nodes colored red and blue accordingly (see~\cref{ex:stationary-er} and~\cref{ex:stationary-pb}).}
    \label{fig:diffusion-on-graph}
\end{figure}

\cref{thm:intro-eigenvalues} is proved in~\cref{sec:walks-on-chambers}.
We also give a sharp bound for the mixing time of this process.

\begin{theorem}[Mixing time for the simple edit process]\label{thm:intro-mixing-time}
    Instate the setup of~\cref{thm:intro-eigenvalues}. Let $m$ denote the number of edges of $\mathcal{H}$ and let $\pi$ denote the stationary distribution of the Markov chain $(G_t)_{t\geq 0}$. Then, the total variation distance of the transition probability matrix $\mathsf{P}$ satisfies
        \begin{align}\label{eq:mixing-time-bound-intro}
            \|\mathsf{P}^t(E_0, \cdot) - \pi\|_{TV}\leq e^{-c}
        \end{align}
    provided $t\geq m(c+2\log{m})$.
\end{theorem}

\cref{thm:intro-mixing-time} is proved in~\cref{sec:walks-on-chambers}. We note that the bound in~\cref{eq:mixing-time-bound-intro} is stronger than some well-known spectral mixing time estimates for irreducible and reversible Markov chains (see, e.g.,~\cite{levin2017markov}), which often depend on the stationary distribution $\pi$.

We also obtain analogous results for compound edit processes, defined below. If $x$ is any graph edit, its \emph{support} $\supp(x)\subseteq\mathcal{E}$ is the subset of edges in the host graph which are edited by $x$ (see~\cref{rmk:quotient-map}).

\begin{definition}[Compound edit process]\label{defn:compound-edit-process}
    Suppose $\mathsf{A}$ is a given set of graph edits and $w\in\mathbb{R}^{\mathsf{A}}$ is a probability distribution on $\mathsf{A}$. Starting with an initial subgraph $G_0 = (\mathcal{V}, E_0)$, for each $t\geq 1$, sample an edit $x\in\mathsf{A}$ according to the distribution $w$ and apply $x$ to the graph $G_{t-1} = (\mathcal{V}, E_{t-1})$. The resulting sequence of graphs $(G_t)_{t\geq 0}$ is said to be a \emph{compound edit process obtained from $\mathsf{A}$ with distribution $w$}. 
\end{definition}

\begin{theorem}[Eigenvalues for compound edit processes]\label{thm:eigenvalues-subsemigroup-intro}
    Let $\mathcal{H}=(\mathcal{V}, \mathcal{E})$ be a fixed host graph, let $\mathsf{A}\subseteq\mathcal{S}$ be a given set of edits, and let $w\in\mathbb{R}^{\mathsf{A}}$ be a given probability distribution on $\mathsf{A}$. Let $\mathsf{P}$ denote the transition probability matrix of the compound edit process obtained from $\mathsf{A}$ with distribution $w$. Then $\mathsf{P}$ is diagonalizable and has eigenvalues indexed by those subsets $X\subseteq\mathcal{E}$ which are formed as unions of sets from $\{\supp(x)\}_{x\in\mathsf{A}}$, together with the empty set, with corresponding eigenvalue
        \begin{align*}
            \lambda_X = \sum_{\substack{x\in\mathsf{A} \\ \supp(x) \subseteq X}} w_x.
        \end{align*}
     The multiplicity of $\lambda_X$ depends on the choice of $\mathsf{A}$ and can be computed from the M\"obius function of the join semilattice generated by $\{\supp(x)\}_{x\in\mathsf{A}}$ and set union (see~\cite[Thm. 1]{brownSemigroupsRingsMarkov2000} and~\cite[Sec. 3.7]{stanley2011enumerative}).
\end{theorem}

This leads to an estimate on the mixing time of a compound edit process as follows.

\begin{theorem}[Mixing time for compound edit processes]\label{thm:mixing-time-subsemigroup-intro}
    Instate the setup of~\cref{thm:eigenvalues-subsemigroup-intro}. Let $0\leq \lambda_\ast < 1$ denote the greatest eigenvalue of $\mathsf{P}$ other than $\lambda = 1$. Then for any initial state $G_0 = (\mathcal{V}, E_0)$, we have
        \begin{align}\label{eq:mixing-time-bound-compound}
            \|\mathsf{P}^t(x_0, \cdot) - \pi\|_{TV}\leq e^{-c}
        \end{align}
    provided $t\geq \frac{m\log{2} + c}{1-\lambda_\ast}$.
\end{theorem}

This paper is organized as follows.~\cref{sec:basic-properties} contains notation, basic definitions, and the derivation of the stationary distribution for the evolving graph process using simple edits. In~\cref{sec:lrbs-edit-semigroup} we briefly mention results in the theory of left regular bands and establish properties of our graph edit semigroup. In~\cref{sec:walks-on-chambers}, we associate the simple edit process to random walks on the chambers of the semigroup, and characterize the eigenvalues of its transition probability matrix. This allows us to give bounds on the mixing time of the process. In~\cref{sec:application-subsemigroups} we consider extensions of this setup to evolving graph processes based on compound edits and investigate two examples of compound edit processes: the Moran forest model and a dynamic random intersection graph model. In~\cref{sec:eigenvectors} we provide a formula for the eigenvectors of the transition probability matrix of the simple edit process and a spectral formula for the commute time between subgraphs of the host graph.~\cref{sec:discussion} includes some discussion and remarks.

\section{Basic properties of the simple edit process}\label{sec:basic-properties}



Below we establish some basic properties of the simple edit process introduced in~\cref{defn:simple-edit-process}.

\begin{lemma}\label{lem:ergodicity-of-gemc}
    Let $\mathcal{H}=(\mathcal{V},\mathcal{E})$ be a given host graph, and let $\mathbf{p} = (p_e)_{e\in\mathcal{E}}$ be given satisfying $0<p_e<1$. Then the simple edit process $(G_t)_{t\geq 0}$ satisfies the following properties:
        \begin{enumerate}[label=\textit{(\roman*)}]
            \item $(G_t)_{t\geq 0}$ is irreducible and aperiodic,
            \item $(G_t)_{t\geq 0}$ has a unique stationary distribution given by 
                \begin{align}\label{eq:formula-for-gemc-stationary}
                    \pi(E) = \left(\prod_{e\in E} p_e\right) \left(\prod_{e\in\mathcal{E}\setminus E} 1-p_e\right),
                \end{align}	
            \item and $(G_t)_{t\geq 0}$ is reversible.
        \end{enumerate}
\end{lemma}

\begin{proof}
    Claim \textit{(i)} follows upon inspection. We establish \textit{(ii)} by direct calculation and show $\pi \mathsf{P} = \pi$, where $\mathsf{P}$ is the transition probability matrix for $(G_t)_{t\geq 0}$ and $\pi$ is as in~\cref{eq:formula-for-gemc-stationary}. To this end, fix $F\in 2^\mathcal{E}$, and for any $e\in {\mathcal{E}}$, let ${e^+}F$ (resp. ${e^-}F$) denote the set obtained by adding (resp. deleting) edge $e$ to (resp. from) $F$. We first note that, based upon the definition of $(G_t)_{t\geq 0}$, the entry of $\pi\mathsf{P}$ corresponding to $F$ can be expressed as follows:
        \begin{align*}
            \left(\pi \mathsf{P}\right)(F) &= \sum_{E\in 2^\mathcal{E}} \pi(E)\mathsf{P}_{EF}\notag\\
            &= \pi(F)\mathsf{P}_{FF} + \sum_{e\in F} \pi{({e^-}F)}\frac{p_e}{{m}}  + \sum_{e\notin F} \pi{({e^+}F)}\frac{1-p_e}{{m}}.
        \end{align*}
    The transition probability from $F$ to itself is given by the sum of the probabilities of adding (resp. removing) each edge which occurs (resp. does not occur) in $F$ already, i.e., we have following formula for the matrix entry $\mathsf{P}_{FF}$:
        \begin{align*}
            \mathsf{P}_{FF} &= \sum_{e\in F} \mathsf{P}_{F, {e^+}F} + \sum_{e\notin F} \mathsf{P}_{F, {e^-}F}.
        \end{align*}
    Thus we have
        \begin{align}\label{eq:transition-matrix-product}
            \left(\pi \mathsf{P}\right)(F)&= \sum_{e\in F} \left\{\pi{({e^-}F)}\frac{p_e}{{m}} + \pi{(F)}\frac{p_e}{{m}}\right\}\notag\\
            &\quad+\sum_{e\notin F} \left\{\pi{(F)}\frac{1-p_e}{{m}} + \pi{({e^+}F)}\frac{1-p_e}{{m}}\right\}.
        \end{align}
    with the convention that $\sum_{\varnothing}(\cdot) = 0$. Next, we observe from~\cref{eq:formula-for-gemc-stationary} that it holds $\pi(e^- F) = \frac{1 - p_e}{p_e} \pi(F)$ whenever $e\in F$ and that $\pi(e^+ F) = \frac{p_e}{1-p_e} \pi(F)$ if $e\notin F$. Thus by substitution into~\cref{eq:transition-matrix-product} we have:
        \begin{align*}
            \left(\pi \mathsf{P}\right)(F)&= \sum_{e\in F} \frac{\pi{(F)}}{{m}} +\sum_{e\notin F}  \frac{\pi{(F)}}{{m}} =\pi(F).
        \end{align*}
    Since \textit{(i)} guarantees the uniqueness of any stationary distribution, the claim follows. The proof of \textit{(iii)} follows from a similar calculation. 
\end{proof}

In the setting where $\mathcal{H}$ is the complete graph $K_n$, we have the following examples which show that many edge-independent random graph models can be realized as the stationary distributions of the simple edit process for various choices of $(p_e)_{e\in\mathcal{E}}$.

\begin{example}[Erd\H{o}s-R\'{e}nyi graphs]\normalfont\label{ex:stationary-er}
    Recall that an Erd\H{o}s-R\'{e}nyi random graph $\mathsf{G}(n, p)$ is constructed by starting with a vertex set $\mathcal{V}$ of cardinality $n$ and independently adding edge $e\in {\mathcal{V} \choose 2}$ to a graph $G$ with probability $p\in(0, 1)$, for each such $e$. The stationary distribution of the simple edit process on $K_n$ with edge probabilities $p_e = p$ for each $e\in {\mathcal{V} \choose 2}$ is equal to the distribution of $\mathsf{G}(n, p)$ on $2^{{\mathcal{V}\choose 2}}$.
\end{example}

\begin{example}[Chung-Lu graphs]\normalfont\label{ex:stationary-cl}
    The Chung-Lu expected degree model, first introduced in~\cite{aiello2001random}, is defined as follows. Fix a vertex set $\mathcal{V}$ and a sequence of positive real numbers $(k_v)_{v\in\mathcal{V}}$ with $\max_{v\in\mathcal{V}}k_v\leq ({\sum_{v\in\mathcal{V}} k_v})^{1/2}$,
    which serves as the expected degree sequence. Starting with no edges, for each pair $e=\{u, v\}$ with $u, v\in\mathcal{V}$, add the edge $e$ to graph with probability
        \begin{align*}
            p_{uv} = \frac{k_u k_v }{\sum_{w\in\mathcal{V}} k_w},
        \end{align*}
    independently of all other edges. Letting $\mathbf{p} = (p_{uv})_{u, v\in\mathcal{V}}$, the stationary distribution of the simple edit process on $K_n$ with edge probabilities $\mathbf{p}$ is equal in distribution to the Chung-Lu model.
\end{example}
    
\begin{example}[Stochastic Block Model]\normalfont\label{ex:stationary-pb}
    We consider the setting of~\cref{ex:stationary-er} in which the probabilities of adding pairs $\{u, v\}\in{\mathcal{V}\choose 2}$ are inhomogeneous in the following sense. Letting $A, B\subseteq \mathcal{V}$ form a disjoint, nonempty partition of the vertex set $\mathcal{V}$, define
        \begin{align*}
            p_{uv} &= \begin{cases}
                p &\text{ if } \{u, v\}\subseteq A\text{ or }\{u, v\}\subseteq B\\
                q &\text{ otherwise,}
            \end{cases},
        \end{align*}
    for $u, v\in\binom{\mathcal{V}}{2}$ and constants $p, q$ satisfying $0<p, q<1$. Letting $\mathbf{p} = (p_{e})_{e\in{\mathcal{V}\choose 2}}$, the stationary distribution of the simple edit process on $K_n$ with edge probabilities $\mathbf{p}$ is equal in distribution to the two-component stochastic block model.
\end{example}

\section{Left regular bands and the graph edit semigroup}\label{sec:lrbs-edit-semigroup}

In this section, we state a number of useful facts about the spectral theory for semigroups, and then apply them to the graph edit semigroup. Our notation and terminology on semigroups largely follow~\cite{brownSemigroupsRingsMarkov2000}.

Recall that a semigroup $\Sigma = (\Sigma, \cdot)$ is a set $\Sigma$ equipped with a binary associative operation $\cdot$ under which it is closed (and which we denote simply by concatenation of elements). We say that $\Sigma$ is a \textit{band} if each element is idempotent, i.e., $x^2 = x$ for each $x\in \Sigma$. We say that $\Sigma$ is a \textit{left regular band} (LRB) if it is a band and satisfies the \textit{memoryless} property, i.e., for each $x,y\in \Sigma$,~\cref{eq:memorylessness} holds. These two properties allow one to easily simplify long products of the form $x_1x_2\cdots x_n$ that may arise in $\Sigma$ by deleting terms in the product which occur more than once on the right. Moreover, it follows that $\Sigma$ is finite if it is finitely generated.

\begin{definition}[Graph edit semigroup]\label{def:graph-edit-semigroup}
    Let $\mathcal{H}=(\mathcal{V},\mathcal{E})$ be a fixed host graph and let $E\subseteq\mathcal{E}$ be a subset of edges. For each $e\in \mathcal{E}$, let the symbol ${e^+}$ denote the operation of ``add edge $e$ to $E$'' and let ${e^-}$ denote the operation of ``remove edge $e$ from $E$''. We call these {\normalfont simple edits}. We write, for $e_i\in {\mathcal{E}}$ and $\sigma_i\in\{\pm\}$,
        \begin{align}\label{eq:edit-introducton}
            e_k^{\sigma_k} e_{k-1}^{\sigma_{k-1}} \cdots e_{1}^{\sigma_1} E
        \end{align}
    to represent the graph obtained by {\normalfont first} performing the edit $e_{1}^{\sigma_1}$, then so on, until performing the edit $e_k^{\sigma_k}$. If $x, y$ are sequences of edits such as those applied to $E$ in~\cref{eq:edit-introducton}, we write $x = y$ if $xE = yE$ for each $E\subseteq\mathcal{E}$. The semigroup $\mathcal{S} = \mathcal{S}(\mathcal{H})$ generated by the set of simple edits is called the {\normalfont graph edit semigroup}. By convention, $\mathcal{S}$ contains an identity edit denoted $\widehat{0}$.
\end{definition}

Next we show that $\mathcal{S}$ is a left regular band.

\begin{lemma}\label{lem:lrb-graph-edits}
    The graph edit semigroup $\mathcal{S}$ is a left regular band.
\end{lemma}

\begin{proof}
    For any $e\in\mathcal{E}$, we have ${e^+}{e^+} = e^+$ and similarly for $e^-$. Thus idempotence extends to any sequence of edits. For the memoryless property, let $y\in \mathcal{S}$ be any edit and observe that the graph ${e^+}y{e^+}E$ is the same as the graph ${e^+} y E$ since the terminal operation ${e^+}$, regardless of the edits prescribed by $y$, returns the edge $e$ to the graph. Thus the initial addition of $e$ becomes redundant. This also extends to products of elements, and we have that $xyx = xy$ holds for any $x, y\in \mathcal{S}$. The claim follows.
\end{proof}

We note that each $x\in \mathcal{S}$ may be written in the form
    \begin{align*}
        x = e_k^{\sigma_k} e_{k-1}^{\sigma_{k-1}} \cdots e_{1}^{\sigma_1}
    \end{align*}
for some finite collection of simple edits $\{e_i^{\sigma_i}\}_{i=1}^k$. We refer to this as an \textit{edit enumeration of} $x$, and among all such enumerations, those with minimal length as \textit{reduced edit enumerations of} $x$. If $e\neq f$ in ${\mathcal{E}}$ and $\sigma,\tau\in\{\pm\}$, then naturally $e^\sigma f^\tau = f^\tau e^\sigma$ so that distinct edges may appear anywhere in the sequence. Consequently, we observe that any such reduced edit enumeration must contain at most one element of the form $e^{(\cdot)}$ for any particular $e\in {\mathcal{E}}$, since ${e^+}{e^-} = {e^+}$ and ${e^-}{e^+} = {e^-}$ for any $e\in\mathcal{E}$. Therefore, each reduced edit enumeration of a given element $x\in \mathcal{S}$ contains the same set of simple edits, possibly appearing in a different order. For $x\in\mathcal{S}$, we write $\simple{x}$ for the set of simple edits obtained from any reduced edit enumeration of $x$.

For a generic LRB $\Sigma$, we can equip $\Sigma$ with the relation $\leqslant$ defined by $x\leqslant y$ if $xy = y$. It is relatively straightforward to show that $\leqslant$ is transitive (for any semigroup) and reflexive (since $\Sigma$ is a band); moreover, $\leqslant$ is antisymmetric and thus defines a partial order on $\Sigma$ owing to the memorylessness property (see~\cite[Proposition 7]{brownSemigroupsRingsMarkov2000}). In the case of the graph edit semigroup $\mathcal{S}$, this partial order can be nicely understood in terms of the edges affected by particular edits. 

\begin{lemma}\label{lem:right-dominance-char}
    Let $x$ and $y$ be two fixed elements of the graph edit semigroup $\mathcal{S}$. Then $x\leqslant y$ if and only if $\simple{x}\subseteq\simple{y}$.
\end{lemma}

\begin{proof}
    Assume $x\leqslant y$. Fix reduced edit enumerations of $x, y$ as follows:
        \begin{align}\label{eq:mars-of-xy}
            x &= e_s^{\sigma_s} e_{s-1}^{\sigma_{s-1}} \cdots e_{1}^{\sigma_1} \notag\\
            y &= f_t^{\tau_t} f_{t-1}^{\tau_{t-1}} \cdots f_{1}^{\tau_1},
        \end{align}
    where $s, t\geq 0$, $e_i, f_i\in {\mathcal{E}}$, and $\sigma_i, \tau_i\in\{\pm\}$.
    Letting the edit $e_i^{\sigma_i}$ be fixed, if $e_i = f_\ell$ for some $\ell$, then it must hold that $\sigma_i = \tau_\ell$. Otherwise if, without loss of generality, we have $\sigma_i=+$ and $E\subseteq\mathcal{E}$, it would hold that $e_i\in xy E$ and $e_i\notin y E$ and thus $xy\neq y$, a contradiction. It follows that $e_i^{\sigma_i}\in\simple{y}$ in this case, and thus we assume that $e_i \notin\{f_\ell\}_{i=1}^t$. Assume again, without loss of generality, that $\sigma_i = +$. Let $E_0=\mathcal{E}\setminus\{e_i\}$ and let $E_1 = yE_0$. Then by the idempotence property, $yE_1 = E_1$, but $xy E_1 = xE_1 \subsetneq E_1$ by choice of $E_0$ so that $xy E_1\neq yE_1$ and therefore $x\nleqslant y$, which is a contradiction. Thus $\simple{x}\subseteq\simple{y}$ holds.
    
    For the converse, assume $\simple{x}\subseteq\simple{y}$ and once again fix reduced edit enumerations of $x, y$ as in~\cref{eq:mars-of-xy}. We have that each $e_i^{\sigma_i}$ appears in the reduced edit enumeration of $y$, and therefore that $e_i^{\sigma_i} y = y$. The claim follows by extension to each simple edit and thus to $x$.
\end{proof}

We can equip a generic LRB $\Sigma$ with the further relation $\prec$ by writing $x\prec y$ if $yx = y$. It is straightforward to show that $\prec$ is transitive and reflexive (in the latter case, since $\Sigma$ is a band), but in general not antisymmetric and therefore not a partial order. Define $x\simeq y$ whenever $x\prec y$ and $y\prec x$ and set $L = \Sigma \backslash \simeq$. Then $L$ is a partially ordered set under $\prec$ and if we denote by $\supp:\Sigma\rightarrow L$ the quotient map, we have that $x\prec y$ holds if and only if $\supp{x} \prec \supp{y}$.

Since $xy \succ x$ and $xy \succ y$ for each $x, y\in \Sigma$ (because of the memoryless and idempotence properties, respectively), we have that $L$ is a join semilattice, and moreover, it can be shown that for each $x, y\in \Sigma$, we have
    \begin{align}\label{eq:supp-property}
        \supp{xy} = \supp{x} \cup \supp{y}.
    \end{align}
Elements $X\in L$ are called \textit{flats} in the hyperplane arrangement literature (see, e.g.,~\cite{brown1998random}). If $\Sigma$ has an identity element (and therefore $L$ has a minimal element) and if $L$ is finite, then $L$ will be a true lattice (see, e.g.,~\cite[Section 3.3]{stanley2011enumerative}). The following lemma completely characterizes this lattice in the case of the graph edit semigroup $\mathcal{S}$.

\begin{lemma}\label{lem:left-dominance-char}
    Let $x$ and $y$ be two fixed elements of the graph edit semigroup $\mathcal{S}$. Then $x\prec y$ if and only if $\supp{x}\subseteq\supp{y}$. The corresponding lattice $\mathcal{L} = \mathcal{S}\backslash \simeq$ can be identified as the Boolean algebra of subsets of edges ${\mathcal{E}}$.
\end{lemma}

\begin{proof}
    This proof follows that of~\cref{lem:right-dominance-char}, with the only distinction being that in order for $yx = y$, it must hold that all of the edges (not including signs) in any reduced edit enumeration of $x$ must also appear in any reduced edit enumeration of $y$. For the latter claim, we therefore have that $yx = y$ and $xy = x$ if and only if $x, y$ consist of simple edits which act on the same subset of edges. That $\mathcal{L}$ is specifically the \emph{Boolean} algebra on these subsets follows from~\cref{eq:supp-property}.
\end{proof}

\begin{remark}\label{rmk:quotient-map}
    It follows from the proof of~\cref{lem:left-dominance-char} that the quotient map $\supp:\mathcal{S}\rightarrow \mathcal{L}$ can be realized as the map which sends each edit to the subset of $\mathcal{E}$ consisting of edges on which the edit acts in a nontrivial manner.
\end{remark}

We say that an element $c\in \Sigma$ is a \textit{chamber} provided that $cx = c$ for each $x\in \Sigma$, or equivalently, $c$ is maximal in the partially ordered set $(\Sigma, \leqslant)$. The set of chambers in an LRB $\Sigma$ forms a two-sided ideal (see, e.g.,~\cite[Proposition 9]{brownSemigroupsRingsMarkov2000}). We denote by $\mathcal{C}$ the ideal of chambers in the graph edit semigroup $\mathcal{S}$. The following lemma characterizes chambers in this setting.

\begin{lemma}\label{lem:chambers-of-s}
    An element $c$ in the graph edit semigroup $\mathcal{S}$ is a chamber if and only if $\supp{c} = {\mathcal{E}}$. In turn, by identifying $c$ with the set of edges $e\in{\mathcal{E}}$ such that ${e^+} \in\simple{c}$, each chamber may be identified as a subset $E\in 2^{\mathcal{E}}$ in a one-to-one and onto fashion.
\end{lemma}

\begin{proof}
    Let $c\in \mathcal{S}$. Then we have that $c$ is a chamber, by definition, whenever $cx = c$ for each $x\in \mathcal{S}$, i.e., whenever $c\succ x$ for each $x$. From~\cref{lem:left-dominance-char}, this holds if and only if $\supp{c}$ contains all possible edges in $\mathcal{E}$. The proof of the latter claim is straightforward and follows along the lines of the proof of~\cref{lem:left-dominance-char}.
\end{proof}

We remark that, in light of~\cref{lem:chambers-of-s}, a random walk on the chambers of $\mathcal{S}$ may be identified as a random walk on labeled subgraphs of $\mathcal{H}$. In~\cref{sec:application-subsemigroups}, we consider the case of semigroups generated by compound edits and random walks on their chambers.

\section{Eigenvalues and mixing times of the simple edit process}\label{sec:walks-on-chambers}

For a LRB $S$ and a probability measure $w \in\mathbb{R}^S$, we define the \textit{semigroup random walk} according to the transition probability matrix $P$ given by
    \begin{align}\label{eq:semi-rw}
        P(x, y) = \sum_{\substack{z\in S\\zx=y}} w_z, \quad x, y\in S.
    \end{align}
Notice that if $C$ is the set of chambers in $S$, then for each $c\in C$, $P(c, x) > 0$ only when there exists $z\in S$ such that $w_z>0$ and $zc=x$. But $C$ is a left ideal so $zc\in C$ and thus $x\in C$ as well. In other words, if and when the random walk reaches the ideal of chambers, it will remain in the ideal for all time. Therefore, rather than studying a Markov chain on the state space $S$ we instead use a state space of the chambers $C$. 

Our main result in this section characterizes the eigenvalues of the transition probability matrix $\mathsf{P}$ corresponding to the simple edit process introduced in~\cref{defn:simple-edit-process}. We note that~\cref{thm:intro-eigenvalues} follows from the following.

\begin{theorem}\label{thm:eigenvalues-state-graph}
    Let $\mathcal{S} = \mathcal{S}(\mathcal{H})$ denote the graph edit semigroup corresponding to the host graph $\mathcal{H}$ as in~\cref{def:graph-edit-semigroup}. Let $(w_x)_{x\in \mathcal{S}}$ be a probability distribution on $\mathcal{S}$, and let $\mathsf{P}$ be the Markov transition probability matrix for the random walk on the chambers of $\mathcal{S}$ given in~\cref{eq:semi-rw}. Then $\mathsf{P}$ is diagonalizable and has an eigenvalue with multiplicity one for each subset $T\subseteq {\mathcal{E}}$ given by
        \begin{align}\label{eq:eigenvalue-01}
            \lambda_T = \sum_{\substack{y\in \mathcal{S} \\\supp{y} \subseteq T}} w_y.
        \end{align}
    In the special case where $(w_x)_{x\in S}$ is nonzero exactly on the set of simple edits with edit probabilities determined by $\mathbf{p} = (p_e)_{e\in\mathcal{E}}$ such that $0<p_e<1$ and which satisfies
        \begin{align*}
            w_{e^+} = 1 - w_{e^-} = \frac{p_e}{{m}},
        \end{align*}
    then the semigroup random walk on the chambers $\mathcal{C}$ of $\mathcal{S}$ has the same transition probability matrix as the simple edit process with edge probabilities $\mathbf{p}$. Moreover, its eigenvalues are indexed by subsets $T\subseteq\mathcal{E}$ and for each such subset we have
        \begin{align}\label{eq:lambda-T}
            \lambda_T = \frac{|T|}{{m}},
        \end{align}
    independent of $(p_e)_{e\in\mathcal{E}}$.
\end{theorem}

\begin{remark}\label{rmk:multiplicities}
    We remark that the eigenvalue $\lambda_T$ in~\cref{thm:eigenvalues-state-graph} that is \textit{indexed by} $T\subseteq\mathcal{E}$ occurs with multiplicity $m_T=1$, but distinct $T$ may give rise to the same value $\lambda\in\mathbb{R}$. In the case of the simple edit process, each eigenvalue of the form $\frac{k}{{m}}$ for some $0\leq k\leq m$ occurs with multiplicity ${{m}\choose k}$.
\end{remark}

\begin{proof}[Proof of~\cref{thm:eigenvalues-state-graph}]
    The main tool that we use here is~\cite[Thm. 1]{brownSemigroupsRingsMarkov2000}, from which it follows that the eigenvalues are as in~\cref{eq:eigenvalue-01}, each with multiplicity $m_X$ satisfying
        \begin{align*}
            \sum_{\substack{Y\in \mathcal{L}\\Y\succ X}} m_Y = c_X
        \end{align*}
    where $c_{\supp{x}}  = |\left\{y\in C: y\geqslant x\right\}|$. Alternatively,
        \begin{align*}
            m_X = \sum_{\substack{Y\in \mathcal{L}\\ Y\geq X}}\mu(X, Y)c_Y
        \end{align*}
    where $\mu$ is the M\"{o}bius function of $L$ (see~\cite[Sec. 3.7]{stanley2011enumerative}). We claim that the multiplicity of each eigenvalue $\lambda_T$ is equal to one. To this end, start by assuming that $|T| = {m}$, i.e., $T = {\mathcal{E}}$. Recall that the set of chambers $\mathcal{C}$ of $\mathcal{S}$ can be identified as consisting of edits whose reduced edit enumerations contain simple edits whose edges exhaust $\mathcal{E}$. Thus we have that by~\cref{lem:right-dominance-char} and~\cite[Thm. 1]{brownSemigroupsRingsMarkov2000}, letting $x\in\mathcal{S}$ be any element with $\supp{x} = \mathcal{E}$ (i.e., $x\in\mathcal{C}$),
        \begin{align*}
            c_{\mathcal{E}=\supp{x}} &= |\left\{y\in \mathcal{C}: \simple{y}\supseteq \simple{x}\right\}| = 1.
        \end{align*}
    since the only other chamber whose reduced edit enumeration subsumes that of a chamber $x$ is $x$ itself. Therefore, $m_\mathcal{E}$ satisfies
        \begin{align*}
            1 = \sum_{\substack{T'\in \mathcal{L} \\ T'\supseteq \mathcal{E}}} m_{T'} = m_\mathcal{E}.
        \end{align*}
    Next, assume for some $1\leq k\leq m$ that $m_Q = 1$ holds for each $Q\subseteq\mathcal{E}$ of size $k\leq Q\leq {m}$ and let $T\subseteq {\mathcal{E}}$ be any fixed subset such that $|T| = k-1$. Then we have for any $x\in \mathcal{S}$ such that $\supp x = T$,
        \begin{align*}
            c_T &= |\left\{y\in \mathcal{C}: \simple{y}\supseteq \simple{x}\right\}| =2^{{m} - (k-1)}
        \end{align*}
    since there are ${m} - (k-1)$ edges \emph{not} belonging to $\supp{x}$, and the chambers in question may be enumerated by concatenating to a reduced edit enumeration of $x$ products of simple edits for each of the remaining edges with any choice of sign. Therefore, $m_T$ satisfies
        \begin{align*}
            2^{{m} - (k-1)} = \sum_{\substack{T'\in \mathcal{L} \\ T'\supseteq T}} m_{T'} = m_T +  \sum_{\substack{T'\in \mathcal{L} \\ T'\supsetneq T}} m_{T'}.
        \end{align*}
    By the induction assumption, we have that $m_{T'}=1$ for each $T'\supseteq T$, and therefore it follows that
        \begin{align*}
            m_T +  \sum_{\substack{T'\in \mathcal{L} \\ T'\supsetneq T}} m_{T'} = m_T + 2^{{m} - (k-1)} - 1,
        \end{align*}
    leaving $m_T=1$ as a result. Our claim that $m_T=1$ for each $T$ follows. For the second half of the claim, we can recognize the semigroup random walk on the chambers of $\mathcal{S}$ as the simple edit process introduced in~\cref{sec:basic-properties} by identifying a chamber $c\in\mathcal{C}$ with 
        \begin{align*}
            E_c = \{e\in \mathcal{E} : e^+\in\simple{c}\}\subseteq\mathcal{E}.
        \end{align*}
    Finally, in this case, if $T\subseteq {\mathcal{E}}$, we have by~\cite[Thm. 1]{brownSemigroupsRingsMarkov2000} that
        \begin{align*}
            \lambda_T &= \sum_{\substack{y\in \mathcal{S} \\ \supp{y}\subseteq T}} w_y = \sum_{e\in T} (w_{{e^+}} + w_{{e^-}})\\
            &= \sum_{e\in T} \left\{\frac{p_e}{{m}} + \frac{1-p_e}{{m}}\right\} = \frac{|T|}{{m}}.
        \end{align*}
    The second claim follows.
\end{proof}

Next, we use the eigenvalues of the simple edit process to deduce bounds on its mixing time. We begin by recalling that if $\mu, \nu \in\mathbb{R}^\mathcal{X}$ are probability distributions on a state space $\mathcal{X}$, their \emph{total variation distance} is given by
    \begin{align}\label{eq:total-variation-dist}
        \|\mu - \nu\|_{\mathrm{TV}} &:= \max_{A\subseteq \mathcal{X}} \left|\sum_{x\in A} \mu(x) - \nu(x)\right|=\frac{1}{2}\sum_{x\in \mathcal{X}} |\mu(x) - \nu(x)|.
    \end{align}
If $(X_t)_{t\geq 0}$ is a Markov chain on a finite state space $\mathcal{X}$ with stationary distribution $\pi$, its mixing time is the least time step $t$ after which $\|P^t(y, \cdot) - \pi\|_{\mathrm{TV}}$ is below some fixed cutoff for each $y\in\mathcal{X}$. In the case of a random walk on the chambers of a left regular band, we recall the following specialized result on the rate of convergence to stationarity.

\begin{theorem}[Theorem 0,~\cite{brownSemigroupsRingsMarkov2000}]\label{thm:convergence-brown}
    Let $S$ be a finite LRB with identity and associated lattice $L$. Let $\{w_x\}_{x\in S}$ be a probability distribution on $S$ such that $S$ is generated by $\{x\in S: w_x>0\}$. Then the random walk on the ideal $C$ of chambers of $S$ (defined by the transition probabilities in~\cref{eq:semi-rw}) has a unique stationary distribution $\pi$, and the total variation distance of the walk from stationarity after starting at an inital chamber $x_0\in C$ satisfies, for each $t>0$,
        \begin{align*}
            \|P^t(x_0, \cdot) - \pi\|_{TV} \leq \sum_{X\in L^\ast} m_X\lambda_X^t
        \end{align*}
    where $L^\ast$ consists of each except the maximal element of $L$ and $\lambda_X$ (resp. $m_X$) denotes the eigenvalue (resp. the multiplicity) of the element $X\in L^\ast$, as in~\cite[Thm. 1]{brownSemigroupsRingsMarkov2000}.
\end{theorem}

We note that a more general version of this theorem appears in~\cite{bidigare1999combinatorial}, and the result itself appears first in the hyperplane walk literature~\cite{brown1998random}. We can apply this directly to the case of the simple edit process, as follows, to establish that the graph edit process mixes in time near-linear in $m$. We note that~\cref{thm:intro-mixing-time} follows from the result below.

\begin{theorem}\label{thm:mixing-time}
    Let $\mathcal{H} = (\mathcal{V},\mathcal{E})$ be a fixed host graph and let $\mathbf{p}=(p_e)_{e\in\mathcal{E}}$ be fixed with $0<p_e<1$ for each $e\in\mathcal{E}$. Let $(G_t)_{t\geq 0}$ be obtained from the simple edit process with edge probabilities $\mathbf{p}$ and initial state $G_0 = (\mathcal{V}, E_0)$ and let $\pi$ denote the corresponding stationary distribution. Then, for each $t\geq 2m\log{m}$, we have
        \begin{align}\label{eq:mixing-time-1}
            \|\mathsf{P}^t(E_0, \cdot) - \pi\|_{TV}\leq 2m\left(1-\frac{1}{m}\right)^t.
        \end{align}
    In particular, if $c>0$ is given, we have
        \begin{align*}
            \|\mathsf{P}^t(E_0, \cdot) - \pi\|_{TV}\leq e^{-c}
        \end{align*}
    provided $t\geq m(c+2\log{m})$.
\end{theorem}

\begin{proof}
    By using~\cref{thm:eigenvalues-state-graph} and~\cref{thm:convergence-brown}, we have
        \begin{align*}
            \|\mathsf{P}^t(E_0, \cdot) - \pi\|_{TV}&\leq \sum_{X\in L^\ast} m_X\lambda_X^t\leq \sum_{k=1}^{m-1}\left(\frac{k}{m}\right)^t{m\choose k}
        \end{align*}
    Let $a_k = \left(\frac{k}{m}\right)^t{m\choose k}$, then we have for $k\geq 1$, 
        \begin{align*}
            \frac{a_k}{a_{k+1}} &= \frac{k^t}{(k+1)^t}\frac{(k+1)}{m-k} \leq m\left(1 - \frac{1}{m}\right)^t\leq m e^{-t/m}
        \end{align*}
    Thus if $t \geq 2m\log{m}$, then $m e^{-t/m}\leq \frac{1}{m} \leq 1/2$. Therefore we have that
        \begin{align*}
            a_k \leq \left(\frac{1}{2}\right)^{m-1-k} a_{m-1},
        \end{align*}
    so that in turn
        \begin{align*}
            \sum_{k=1}^{m-1}\left(\frac{k}{m}\right)^t{m\choose k} &\leq m\left(1-\frac{1}{m}\right)^t\left(\sum_{k=1}^{m-1} \left(\frac{1}{2}\right)^{m-1-k}\right)\\
            &\leq 2m\left(1-\frac{1}{m}\right)^t
        \end{align*}
    and~\cref{eq:mixing-time-1} is proved. Note that
        \begin{align*}
            2m\left(1-\frac{1}{m}\right)^t \leq e^{-c}
        \end{align*}
    holds provided
        \begin{align*}
            t&\geq -\frac{c+\log{2m}}{\log{1-1/m}}.
        \end{align*}
    Since $\log(1+x)\leq x$ for $x>-1$, we have
        \begin{align*}
            -\frac{c+\log{2m}}{\log{1-1/m}} &\leq m(c+\log{2m}) \leq m(c+2\log{m}),
        \end{align*}
    where the last estimate is taken to ensure $t\geq 2m\log{m}$ as well. The theorem is proved.
\end{proof}
	
\section{Compound edit processes and the Moran forest}\label{sec:application-subsemigroups}

In this section we consider the setting of compound edit processes, which were introduced in~\cref{defn:compound-edit-process}. In this setting, the resulting evolving processes are considerably more complex when compared to the simple edit process considered in earlier sections. Nevertheless, our methods can be used to obtain closed-form formulas for the eigenvalues of the transition probability matrices as well as bounds for the corresponding Markov chain mixing times. We consider two examples of compound edit processes which are motivated by previous work in population biology modeling and random graphs, respectively, and compute their eigenvalues. We remark that our methods can be used to analyze further processes such as population dynamic models, birth-death processes, or contact-diffusion models.

Elements of the graph edit semigroup $\mathcal{S}$ are called \emph{compound edits}, which are products of simple edits. Thus compound edit processes can be viewed as random walks on the chambers of subsemigroups of $\mathcal{S}$ which are generated by specified families of compound edits and we may thereby apply much of the setup from~\cref{sec:lrbs-edit-semigroup}. We note that the chambers of semigroups generated by compound edits will form a subset of the chambers of the larger graph edit semigroup. The following lemma characterizes the join semilattice associated to such subsemigroups. 

\begin{lemma}\label{lem:subsemigroup}
    Let $\mathsf{A}\subseteq\mathcal{S}$ be a given set of compound edits belonging to the graph edit semigroup $\mathcal{S}$ and let $\mathcal{T}=\langle \mathsf{A}\rangle$ be the subsemigroup of $\mathcal{S}$ generated by the edits $\mathsf{A}$. Then the join semilattice $\mathcal{L}(\mathcal{T})$ associated to $\mathcal{T}$ (as in~\cref{lem:left-dominance-char}) is generated by the sets $\{\supp(x)\}_{x\in\mathsf{A}}$ and the set union operation.
\end{lemma}

Equivalently, $\mathcal{L}(\mathcal{T})$ may be described as the smallest subset of $2^\mathcal{E}$ containing the sets $\varnothing, \{\supp(x)\}_{x\in\mathsf{A}}$ and which is closed under the map $X,Y\mapsto X\cup Y$. The proof of this result is essentially identical to the proof of~\cref{lem:left-dominance-char} and is omitted. Next we characterize the eigenvalues of the random walk on the chambers of a subsemigroup $\mathcal{T}$ as in~\cref{lem:subsemigroup}.

\begin{theorem}\label{thm:eigenvalues-subsemigroup}
    Let $\mathcal{H}$ denote a fixed host graph and let $\mathsf{A}\subseteq\mathcal{S}$ be a given set of edits, let $w\in\mathbb{R}^{\mathsf{A}}$ be a given probability distribution on $\mathsf{A}$, and let $\mathcal{T}=\langle \mathsf{A}\rangle$ be the subsemigroup of $\mathcal{S}$ generated by the edits $\mathsf{A}$, and which by convention contains the identity edit $\widehat{0}$. Consider the random walk on the chambers of $\mathcal{T}$ with transition probability matrix $\mathsf{P}$ given in~\cref{eq:semi-rw}. Then $\mathsf{P}$ coincides with the transition probability matrix of the compound edit process obtained from $\mathsf{A}$ with distribution $w$. Moreover, $\mathsf{P}$ is diagonalizable and has eigenvalues indexed by elements $X\in \mathcal{L}(\mathcal{T})$, each of which is identified as a subset of $\mathcal{E}$, with corresponding eigenvalue
        \begin{align*}
            \lambda_X = \sum_{\substack{x\in\mathsf{A} \\ \supp(x) \subseteq X}} w_x.
        \end{align*}
     The multiplicity of $\lambda_X$ depends on the choice of $\mathsf{A}$ and can be computed from the M\"obius function of the join semilattice generated by $\{\supp(x)\}_{x\in\mathsf{A}}$ and set union (see~\cite[Thm. 1]{brownSemigroupsRingsMarkov2000} and~\cite[Sec. 3.7]{stanley2011enumerative}).
\end{theorem}

We note that the number of eigenvalues of $\mathsf{P}$ and their multiplicities depend significantly on the choice of generating family $\mathsf{A}$ and thus do not admit a general closed-form enumeration. The proof is essentially the same as the first half of that of~\cref{thm:eigenvalues-state-graph} and is omitted.

By~\cref{thm:convergence-brown}, we know that the random walk on the chambers of $\mathcal{T}$ will be ergodic and its mixing time can be bounded by the eigenvalues $\lambda_X$ of $\mathsf{P}$. Without the multiplicities of the eigenvalues, our mixing time bounds (cf.~\cref{thm:mixing-time}) are different, but we can still guarantee a mixing time which is polynomial in $m$ provided each of the weights $w_x$ in~\cref{thm:eigenvalues-subsemigroup} are bounded from below by the reciprocal of a polynomial in $m$. We make this precise below.

\begin{theorem}\label{thm:mixing-time-subsemigroup}
    Let $\mathsf{A}\subseteq\mathcal{S}$ be a given family of edits, let $\mathcal{T}=\langle \mathsf{A}\rangle$ be the subsemigroup of $\mathcal{S}$ generated by $\mathsf{A}$, let $w\in\mathbb{R}^{\mathsf{A}}$ be a given probability distribution on $\mathsf{A}$, and let
        \begin{align*}
            \lambda_\ast = \sup_{ \substack{X\in\mathcal{L}(\mathcal{T}) \\ X\neq \mathcal{E}}} \lambda_X
        \end{align*}
    as in~\cref{thm:eigenvalues-subsemigroup}. Consider the compound edit process $(G_t)_{t\geq 0}$ obtained from $\mathsf{A}$ with distribution $w$, and let $\mathsf{P}$ denote its transition probability matrix and let $\pi$ denote its stationary distribution. Then for any initial state $G_0 = (\mathcal{V}, E_0)$ belonging to the chambers of $\mathcal{T}$ and $c>0$, we have
        \begin{align}\label{eq:mixing-time-bound-compound}
            \|\mathsf{P}^t(x_0, \cdot) - \pi\|_{TV}\leq e^{-c}
        \end{align}
    provided $t\geq \frac{m\log{2} + c}{1-\lambda_\ast}$.
\end{theorem}

\begin{proof}
    By~\cref{thm:convergence-brown}, we have that
        \begin{align*}
            \|\mathsf{P}^t(x_0, \cdot) - \pi\|_{TV} \leq \sum_{X\in L^\ast} m_X\lambda_X^t.
        \end{align*}
    We bound the right-hand side as follows. First, the largest eigenvalue $\lambda_X$ must be at most $\lambda_\ast$; and second, the number of chambers of $\mathcal{T}$ must be at most the number of chambers of the entire graph edit semigroup $\mathcal{S}$, which is $2^m$. Thus
        \begin{align*}
            \sum_{X\in L^\ast} m_X\lambda_X^t &\leq 2^m\,\lambda_\ast ^t.
        \end{align*}
    Therefore, by straightforward manipulation, we have that $2^m\,\lambda_\ast ^t\leq e^{-c}$ if and only if 
        \begin{align*}
            t\geq -\frac{m\log{2} + c}{\log(\lambda_\ast)},
        \end{align*}
    and thus by using the bound $\log{1-x}\leq -x$ for $x>-1$, we have that
        \begin{align*}
            -\frac{m\log{2} + c}{\log(\lambda_\ast)} &\leq \frac{m\log{2} + c}{1-\lambda_\ast},
        \end{align*}
    from which the claim follows.
\end{proof}

\begin{remark}\label{rmk:sharpen-mixing-time-chambers}
    We remark that if the cardinality of the set of chambers of the semigroup $\mathcal{T}$ in~\cref{thm:mixing-time-subsemigroup} can be bounded from above by $M\geq 0$,~\cref{eq:mixing-time-bound-compound} can be sharpened to $t\geq \frac{\log{M} + c}{1-\lambda_\ast}$.
\end{remark}

To conclude this section we consider two examples of compound edit processes, the Moran forest model and a dynamic random intersection graph model. 

\begin{figure}[t]
    \centering
    \includegraphics[width=0.75\linewidth]{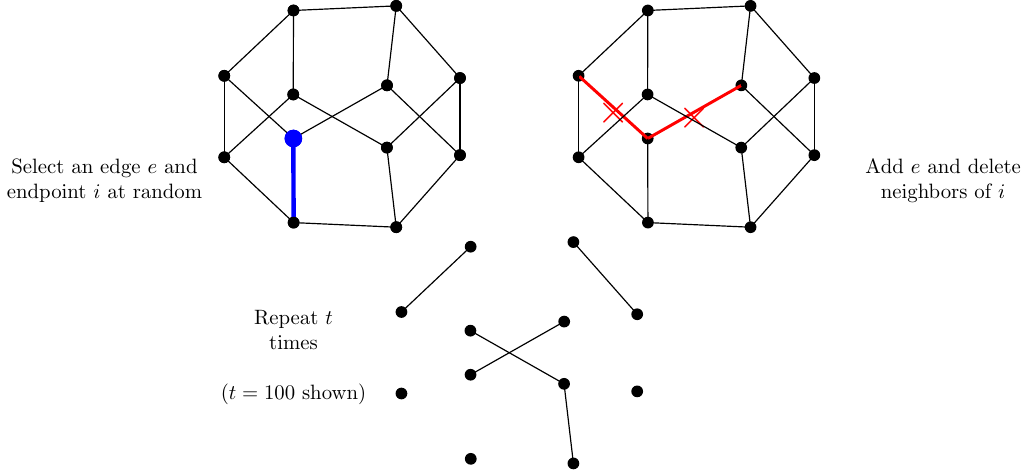}
    \caption{An illustration of the Moran edit process with host graph $\mathcal{H}$ on twelve vertices.}\label{fig:moran}
\end{figure}

\begin{example}[The Moran forest model]\label{ex:moran-forest}\normalfont
	Let $\mathcal{H}=(\mathcal{V},\mathcal{E})$ be a given host graph and set $G_0=\mathcal{H}$. For $t\geq 0$, sample an edge $e=\{u, v\}\in \mathcal{E}$ uniformly at random and an endpoint $x$ uniformly at random from $e$. Let $G_{t+1} = (\mathcal{V}, E_{t+1})$ be obtained by deleting all edges from $E_{t}$ which are incident to the node $x$ and adding the edge $e$ to $E_t$. If $(G_t)_{t\geq 0}$ is a sequence of random graphs obtained from this model, we say $(G_t)_{t\geq 0}$ is obtained from the Moran edit process. The stationary distribution $\pi$ obtained from $\mathcal{H} = K_n$ is known as the \emph{Moran forest model}~\cite{bienvenuMoranForest2021}. We illustrate an example of this Markov chain in~\cref{fig:moran}.

    It can be shown without any semigroup theory that this process is ergodic and thus has a unique stationary distribution (see, e.g.,~\cite{bienvenuMoranForest2021}); it has been used as a model for evolutionary processes and in particular can be used to model the family structure of a random population obtained from the Moran model (see~\cite{durrett2008probability}). The state graph associated to this walk in the case where $\mathcal{H}=K_4$ is shown in~\cref{fig:state-graph-moran}.

	We can recast this process as a particular instance of an edit semigroup random walk as follows. We begin by letting $\mathcal{E}' = \{(u, v), (v, u) : \{u, v\}\in \mathcal{E}\}$ denote the oriented edges of the host graph $\mathcal{H}$. For each $(u, v)\in \mathcal{E}'$, write
		\begin{align*}
			y_{(u, v)} = \{u, v\}^+ \left(\prod_{\substack{e\in \mathcal{E} \\ u\in e}} e^{-}\right),
		\end{align*}
	where we use the notation $\{u, v\}^+$ to refer to the edit which adds the edge $\{u, v\}$ to a given subgraph. We define
		\begin{align*}
			\mathsf{B} = \{y_{(u, v)} : (u, v) \in \mathcal{E}'\},
		\end{align*}
	and denote by $\mathcal{M}=\langle \mathsf{B}\rangle$ the semigroup which is generated by $\mathsf{B}$. In the setup of~\cref{thm:eigenvalues-subsemigroup}, by defining for each $v\in \mathcal{V}$ the set $N(v):=\{e\in \mathcal{E} : v\in e\}$, it follows by~\cref{lem:subsemigroup} that the join semilattice $\mathcal{L}(\mathcal{M})$ is the smallest subset of $2^\mathcal{E}$ containing the neighborhoods $\{N(v)\}_{v\in\mathcal{V}}$ and $\varnothing$ and which is closed under set unions. Moreover, by~\cref{thm:eigenvalues-subsemigroup}, it follows that the eigenvalues $\lambda_X$ of the transition probability matrix $\mathsf{P}$ of the Moran edit process are indexed by subsets $X\in \mathcal{L}(\mathcal{M})$ and are given by
        \begin{align*}
            \lambda_X = \sum_{\substack{v\in \mathcal{V} \\ N(v)\subseteq X}}\frac{d_v}{2 m},
        \end{align*}
    where $d_v$ denotes the degree of vertex $v\in\mathcal{V}$ in the host graph. It follows that the second largest eigenvalue of $\mathsf{P}$ is at most $1-\frac{\delta(\mathcal{H})}{m}$ where $\delta(\mathcal{H})$ is the minimum vertex degree of $\mathcal{H}$. By~\cref{thm:mixing-time-subsemigroup}, 
        \begin{align*}
            \|\mathsf{P}^t(G_0 = \mathcal{H}, \cdot) - \pi\|_{TV}\leq e^{-c}
        \end{align*}
    holds provided
        \begin{align}\label{eq:moran-forest-mixing-time}
            t\geq \frac{m^2\log(2) + cm}{\delta(\mathcal{H})}.
        \end{align}
    In the case where $\mathcal{H} = K_n$ for some $n\geq 2$, by~\cref{rmk:sharpen-mixing-time-chambers},~\cref{eq:moran-forest-mixing-time} can be sharpened using the bound $F_n\leq (n-1)^{n} \leq e^{n\log n}$ where $F_n$ is the number of spanning forests of $K_n$ (see~\cite[Thm. 1.4]{bencs2023upper}), so that it suffices to pick
        \begin{align}\label{eq:moran-forest-mixing-time-Kn}
            t\geq \frac{n^2\log{n} + cn}{2}.
        \end{align}
\end{example}

\begin{example}[Dynamic random intersection graph]\label{ex:random-intersection}\normalfont
    Consider a set of $n\geq 1$ ground symbols $\Omega = \{1, 2,\dotsc, n\}$. An \emph{intersection graph} has vertex set $V$ which is a family of subsets of $\Omega$ and edges $\{u, v\}$ whenever $u\cap v \neq\varnothing$. An intersection graph is also understood via its bipartite incidence graph, which consists of vertex set $\Omega \cup V$ with edges $\{v, B\}$ for $v\in \Omega$ and $B\in V$ present whenever $v\in B$.
    
    Various examples of random intersection graphs have appeared in the literature dating back to the 1996 thesis of Singer~\cite{singer1996random}. Here we consider a dynamic random intersection graph model (not to be confused with~\cite{milewska2025dynamic}) with stationary distribution equal to the random intersection graph model of Godehardt and Jaworski~\cite{godehardt2001two}, as follows.
    
    Let $\Omega' = \{1, 2,\dotsc, N\}$ for $N\geq 1$ be fixed. Let the host graph $\mathcal{H}$ be the complete bipartite graph on $\Omega\cup \Omega'$. For $v\in \Omega$ and $A\subseteq \Omega'$ fixed, define the compound graph edit
        \begin{align*}
            y_{v, A} &= \prod_{\substack{u\in\Omega' \\ u\in A}} \{u, v\}^+ \prod_{\substack{u\in\Omega' \\ u\notin A}} \{u, v\}^-
        \end{align*}
    That is, $y_{v, A}$ is the edit which assigns the neighborhood of $v$ to consist exactly of the elements of $A$. Note that the support of $y_{v, A}$ consists of the $N$ edges in $\mathcal{H}$ with endpoint $v$. We consider the set of compound edits $\mathsf{B} = \{y_{v, A}\}_{v\in \Omega, A\subseteq\Omega'}$. 

    Next, we fix a probability distribution $\mu\in\mathbb{R}^{N+1}$ on the set $\{0, 1,\dotsc, N\}$ which are understood as the possible cardinalities of neighborhoods of vertices in $\Omega$. Define a probability distribution $w$ on $\mathsf{B}$ as follows:
        \begin{align*}
            w({y_{v, A}}) &= \frac{1}{n}\frac{\mu(|A|)}{\binom{N}{|A|}},\quad v\in\Omega,\;A\subseteq\Omega'.
        \end{align*}
    That is, $w_{v, A}$ may be understood as the measure of a uniformly selected vertex $v$ and a subset $A$ sampled first by its cardinality according to $\mu$ and then uniformly among all subsets of $\Omega'$ with corresponding cardinality. The compound edit process on the host graph $\mathcal{H}$ with edits $\mathsf{B}$ and distribution $w$ can be thought of as iteratively selecting a vertex $v\in\Omega$ at random and then resetting its neighbors in $\Omega'$ according to the update rule given by $\mu$. It is straightforward to verify that its stationary distribution agrees with the random intersection graph model defined by $n, N, \mu$ (see~\cite{godehardt2001two}). 

    Assuming $\mu$ has full support, the set of chambers of the associated subsemigroup $\mathcal{I} = \langle \mathsf{B}\rangle$ is the same as the entire edit semigroup and has cardinality $2^{nN}$. Moreover, the eigenvalues of this walk are indexed by elements of the join semilattice generated by the edge neighborhoods of each vertex $v\in\Omega$. Since these sets are in fact disjoint, for each subset $B\subseteq \Omega$, there is an eigenvalue
        \begin{align*}
            \lambda_B = \sum_{v\in B} \sum_{\substack{y\in\mathsf{B} \\ \supp(x) = N(v)}} w(y) = \frac{|B|}{n},
        \end{align*}
    so that the spectral gap is $1/n$ independent of the choice of $\mu$. By~\cref{thm:mixing-time-subsemigroup} and~\cref{rmk:sharpen-mixing-time-chambers}, letting $\mathsf{P}$ (resp. $\pi$) denote the transition probability matrix (resp. stationary distribution) for the compound edit process, for any given initial state $G_0 = (\Omega\cup\Omega', E_0)$,
        \begin{align*}
            \|\mathsf{P}^t(E_0, \cdot) - \pi\|_{TV}\leq e^{-c}
        \end{align*}
    holds provided $t\geq Nn^2\log{2} + cn$.
\end{example}

\section{Eigenvectors of the graph edit Markov chain and commute times}\label{sec:eigenvectors}

In this section we obtain closed-form expressions for each of the eigenvectors of the transition probability matrix $\mathsf{P}$ of the simple edit process. The main tool utilized in discovering the eigenvectors comes from the recursive procedure introduced by Saliola in the paper~\cite{saliola2012eigenvectors}; we mention this context for completeness, although the argument presented herein is self-contained. 

\begin{theorem}\label{thm:eigenvectors}
    Let $\mathcal{H}=(\mathcal{V}, \mathcal{E})$ be a given host graph and let $\mathsf{P}$ denote the transition probability matrix of the simple edit process with edge probabilities $\mathbf{p} = (p_e)_{e\in\mathcal{E}}$. Letting $T\subseteq\mathcal{E}$ be any subset of edges, define the row vector $\phi_T\in\mathbb{R}^{2^\mathcal{E}}$ componentwise via the formula
        \begin{align}\label{eq:eigenvector-defn}
            \phi_T(E) &= (-1)^{|\mathcal{E}\setminus(E\cup T)|}\prod_{e\in T\cap E}p_e \prod_{e\in T\setminus E}(1-p_e),\quad E\subseteq \mathcal{E}.
        \end{align}
    Then $\phi_T$ satisfies $\phi_T \mathsf{P} = \frac{|T|}{m}\phi_T$.
\end{theorem}

\begin{proof}
    We can prove this using an inductive argument on the size of the subset $T$. For $|T| = |\mathcal{E}|$ we know that the claim is true since $\phi_\mathcal{E} = \pi$ as in~\cref{lem:ergodicity-of-gemc}. Assume the claim is true for each such $T$ with $|T|\geq k$ where $1\leq k \leq |\mathcal{E}|$, and fix a subset $T\subseteq \mathcal{E}$ with $|T| = k-1$. We write $T = \{t_1,\dotsc, t_{k-1}\}$ for $t_i\in\mathcal{E}$ and put
        \begin{align*}
            T^\ast = \{t_1,\dotsc, t_{k-1}, t_\ast\}
        \end{align*}
    where $t_\ast\in\mathcal{E}$ is any element which does not originally appear in $T$. Thus $T = T^\ast\setminus\{t_\ast\}$ and for each $E\subseteq\mathcal{E}$ it holds
        \begin{align}\label{eq:phi-tstar-defn}
            (\phi_{T^\ast} \mathsf{P})(G) &=\frac{|T^\ast|}{m}\phi_{T^\ast}(G).
        \end{align}
    Moreover, we have upon inspection that
        \begin{align}\label{eq:phi-tstar-prop2}
            \phi_T(E) &= \begin{cases}
                \frac{1}{p_{t_\ast}} \phi_{T^\ast}(E) &\text{ if }t_\ast\in E\\
                -\frac{1}{1-p_{t_\ast}}\phi_{T^\ast}(E) &\text{ if }t_\ast\notin E\\
            \end{cases}.
        \end{align}
    Now fix any $E\subseteq\mathcal{E}$ and assume $t_\ast\in E$. We observe that 
        \begin{align}\label{eq:phi-tstar-prop3}
            \phi_{T^\ast}(t_\ast^{-}E) &= \frac{1-p_{t_\ast}}{p_{t_\ast}}\phi_{T^\ast}(E).
        \end{align}
    We can proceed with the calculation of $\phi_{T}\mathsf{P}$ (using an argument along the lines of that in the proof of~\cref{lem:ergodicity-of-gemc}) as follows:
        \begin{align*}
            (\phi_T \mathsf{P})(E) &=\sum_{e\in E}\left\{\phi_T(e^- E)\frac{p_e}{m} + \phi_T(E)\frac{p_e}{m}\right\} \\
            &\quad+ \sum_{e\notin E} \left\{\phi_T(e^+ E)\frac{1-p_e}{m} + \phi_T(E)\frac{1-p_e}{m}\right\}\\
            &=\frac{1}{p_{t_\ast}}\biggl [\sum_{e\in E} \left\{\phi_{T^\ast}(e^- E)\frac{p_e}{m} + \phi_{T^\ast}(E)\frac{p_e}{m}\right\}  \\
            &\quad + \sum_{e\notin E} \left\{\phi_{T^\ast}(e^+ E)\frac{1-p_e}{m} + \phi_{T^\ast}(E)\frac{1-p_e}{m}\right\}\biggr ] \\
            &\quad + \phi_{T}(t_\ast^- E)\frac{p_{t_\ast}}{m} - \frac{\phi_{T^\ast}(t_\ast^- E)}{p_{t_\ast}} \frac{p_{t_\ast}}{m}\\
            &= \frac{|T^\ast|}{m p_{t_\ast}}\phi_{T^\ast}(E)+ \phi_{T}(t_\ast^- E)\frac{p_{t_\ast}}{m} - \frac{\phi_{T^\ast}(t_\ast^- E)}{p_{t_\ast}} \frac{p_{t_\ast}}{m}
        \end{align*}
    using~\cref{eq:phi-tstar-defn} and~\cref{eq:phi-tstar-prop2}. Using ~\cref{eq:phi-tstar-prop2} again and~\cref{eq:phi-tstar-prop3}, it follows that
        \begin{align*}
            &\frac{|T^\ast|}{m p_{t_\ast}}\phi_{T^\ast}(E)+ \phi_{T}(t_\ast^- E)\frac{p_{t_\ast}}{m} - \frac{\phi_{T^\ast}(t_\ast^- E)}{p_{t_\ast}} \frac{p_{t_\ast}}{m}\\
            &\quad=\frac{|T^\ast|}{m p_{t_\ast}}\phi_{T^\ast}(E)- \frac{\phi_{T^\ast}(E)}{1-p_{t_\ast}}\frac{1-p_{t_\ast}}{p_{t_\ast}} \frac{p_{t_\ast}}{m}\\
            &\quad\quad- \frac{\phi_{T^\ast}(E)}{p_{t_\ast}} \frac{p_{t_\ast}}{m} \frac{1-p_{t_\ast}}{p_{t_\ast}}\\
            &\quad= \frac{\phi_{T^\ast}(E)}{p_{t_\ast}}\left\{\frac{|T^\ast| - p_{t_\ast} - 1 + p_{t_\ast}}{m}\right\}\\
            &\quad= \phi_{T}(E)\frac{|T^\ast| - 1}{m}  = \phi_T(E) \frac{|T|}{m}
        \end{align*}
    as claimed. The case $t\notin E$ follows similarly, and thus by induction, the claim follows.
\end{proof}

The next result is a technical step which confirms that the family $\phi_T$ exhibited in~\cref{thm:eigenvectors}, when suitably normalized, forms an orthonormal family in $\mathbb{R}^{2^\mathcal{E}}$.

\begin{theorem}\label{thm:orthonormality}
    Let $\mathcal{H}=(\mathcal{V}, \mathcal{E})$ be a given host graph and let $\mathsf{P}$ (resp. $\pi$) denote the transition probability matrix (resp. stationary distribution) of the simple edit process with edge probabilities $\mathbf{p} = (p_e)_{e\in\mathcal{E}}$. For each $T\subseteq\mathcal{E}$, let $\phi_T$ denote the correspondingly indexed eigenvector of $\mathsf{P}$ as in~\cref{eq:eigenvector-defn}. Let $\psi_T\in\mathbb{R}^{2^\mathcal{E}}$ be given by the formula
        \begin{align*}
            \psi_T(E) &=\phi_T(E)  \frac{\prod_{e\notin T} \sqrt{p_e(1-p_e)}}{\sqrt{\pi(E)}},\quad E\subseteq\mathcal{E}.
        \end{align*}
    Then $\{\psi_T\}_{T\subseteq \mathcal{E}}$ form an orthonormal system of (left) eigenvectors for the symmetric matrix $\mathsf{Q} = \Pi^{1/2}\mathsf{P}\Pi^{-1/2}$, where $\Pi = \mathrm{diag}(\pi(E))_{E\subseteq\mathcal{E}}$.
\end{theorem}

\begin{proof}
    Note that, up to the scaling factor $\prod_{e\notin T} \sqrt{p_e(1-p_e)}$ which is uniform in $E\subseteq{\mathcal E}$, it holds $\psi_T = \phi_T \Pi^{-1/2}$. Thus $\psi_T \mathsf{Q}= \phi_T\mathsf{P}\Pi^{-1/2} = \lambda_T \psi_T$, and it follows that $\{\psi_T\}_{T\subseteq \mathcal{E}}$ forms a complete collection of eigenvectors of $\mathcal{Q}$. 

    The symmetry of $\mathsf{Q}$ guarantees orthogonality of $\psi_T, \psi_{T'}$ provided $\lambda_T\neq\lambda_{T'}$; however the spectrum of $\mathsf{P}$ will in general contain many repeated eigenvalues (see~\cref{rmk:multiplicities}) and thus we check orthogonality in general as follows. For a given subset $E\subseteq\mathcal{E}$, let $\mathbf{1}_E\in\mathbb{R}^{\mathcal{E}}$ denote the indicator vector of $E$. Then for any $T\subseteq\mathcal{E}$ fixed, the vector $\phi_T$ may be recast as a product over edges in the form
        \begin{align*}
            \phi_T(E)=\prod_{e\in\mathcal E}g_{e,T}(\mathbf{1}_E(e)),\quad E\subseteq\mathcal{E},
        \end{align*}
    where $g_{e,T}:\{0, 1\}\rightarrow\mathbb{R}$ is given by
        \begin{align*}
            g_{e,T}(x)=
            \begin{cases}
            p_e,&x=1\text{ and }e\in T,\\
            1-p_e,&x=0\text{ and }e\in T,\\
            1,&x=1\text{ and }e\notin T,\\
            -1,&x=0\text{ and }e\notin T.
            \end{cases}.
        \end{align*}
    Hence, for $S, T\subseteq\mathcal{E}$ fixed, it holds
        \begin{align*}
            \langle\psi_T,\psi_S\rangle
            &=\sum_{E\subseteq\mathcal E}\psi_T(E)\,\psi_S(E)\\
            &=\prod_{e\notin T}\sqrt{p_e(1-p_e)}
            \;\prod_{e\notin S}\sqrt{p_e(1-p_e)}\\
            &\quad\times \sum_{X\in\{0,1\}^{\mathcal E}}
            \frac{\prod_{e}g_{e,T}(X(e))\,g_{e,S}(X(e))}{\pi(E)},
        \end{align*}
    where we identify subsets with their indicator vectors. Focusing on the last term for a moment, we have the factorization:
        \begin{align*}
            \sum_{X\in\{0,1\}^{\mathcal E}}&\frac{\prod_{e\in\mathcal{E}}g_{e,T}(X(e))\,g_{e,S}(X(e))}{\pi(E)}\\
            &=\sum_{X\in\{0,1\}^{\mathcal E}} \prod_{e\in\mathcal{E}} \frac{g_{e,T}(X(e)) g_{e,S}(X(e))}{p_e^{X(e)}(1-p_e)^{1-X(e)}}\\
            &=\prod_{e\in\mathcal E}
            \Bigl(\sum_{x=0}^1\frac{g_{e,T}(x)\,g_{e,S}(x)}{p_e^x(1-p_e)^{1-x}}\Bigr).
        \end{align*}
    This can be seen by expanding $\sum_{X\in\{0,1\}^{\mathcal E}}(\cdot)$ into $|\mathcal{E}|$ nested sums and then rearranging. If $T\neq S$, choose an edge $e^\ast$ on which they differ; assuming without loss of generality that $e^\ast\in T$ and $e^\ast\notin S$, one has
        \begin{align*}
            \sum_{x=0}^1\frac{g_{e^\ast,T}(x)\,g_{e^\ast,S}(x)}{p_{e^\ast}^x(1-p_{e^\ast})^{1-x}}
            &=-\frac{1-p_e}{1-p_e} + \frac{p_e}{p_e} = 0.
        \end{align*}
    so $\langle\psi_T,\psi_S\rangle=0$. On the other hand, if $T=S$, then for each $e$ one checks
        \begin{align*}
        \sum_{x=0}^1\frac{g_{e,T}(x)^2}{p_e^x(1-p_e)^{1-x}}
        &=
        \begin{cases}
        \frac{(1-p_e)^2}{1-p_e}+\frac{p_e^2}{p_e}=1,&e\in T,\\
        \frac{1}{1-p_e}+\frac{1}{p_e}=\frac{1}{p_e(1-p_e)},&e\notin T.
        \end{cases}
        \end{align*}
    Therefore
        \begin{align*}
        \langle\psi_T,\psi_T\rangle
        &=\prod_{e\notin T}p_e(1-p_e)
        \;\prod_{e\notin T}\frac1{p_e(1-p_e)}
        =1,
        \end{align*}
    which completes the proof.
\end{proof}

Next we highlight an application to spectral representations of commute times for the simple edit process with edge probabilities $\mathbf{p} = (p_e)_{e\in\mathcal{E}}$. To this end, consider a finite state space $\mathcal X=\{1,\dots,N\}$ and a Markov chain $(X_t)_{t\geq 0}$ on $\mathcal X$. For $x\in\mathcal{X}$, we define the (possibly infinite) \emph{hitting time}
    \begin{align*}
        \tau_x &= \inf \{t\geq 0 : X_t = x\}.
    \end{align*}
The \emph{expected hitting times} and \emph{expected commute times} are given, respectively, by
    \begin{align*}
        H(x, y) &= \mathbb{E}\left[\tau_y \mid X_0 = x\right],\\
        C(x, y) &= H(x, y) + H(y, x),
    \end{align*}
for each $x,y\in\mathcal{X}$. We remark that in the case of our evolving graph processes, the commute time can be considered a notion of distance between subgraphs the host graph, and is related to the effective resistance distance on the state graph. For completeness we recall the following lemma, which is well known and stated without proof.

\begin{lemma}\label{lem:commute}
    Let $\mathcal X=\{1,\dots,N\}$ be a finite state space and let $P$ be the transition probability matrix of an ergodic reversible Markov chain $(X_t)_{t\geq 0}$ on $\mathcal X$, and let $\pi$ be its stationary distribution. Set $\Pi=\operatorname{diag}(\pi(1),\dots,\pi(N))$ and $Q=\Pi^{1/2} P \Pi^{-1/2}$. Assume $Q$ admits the the row vector eigendecomposition 
        \begin{align*}
            Q = \Psi^\top \Lambda \Psi, \quad 
            \Psi = \begin{bmatrix}
                \psi_1 \\
                \psi_2\\
                \vdots\\
                \psi_N
            \end{bmatrix},\quad
            \psi_1=\Pi^{1/2}\mathbf 1
        \end{align*}
    where $1=\lambda_1>\lambda_2\ge\dots\ge\lambda_N>-1$, and define the normalized vectors $\varphi_k=\Pi^{-1/2}\psi_k$ so $\varphi_k(i)=\psi_k(i)/\sqrt{\pi(i)}$. Then for $x, y\in\mathcal{X}$, fixed, the hitting time $H(x, y)$ admits the following spectral representation:
        \begin{align}\label{eq:hitting-time-spectral}
            H(x, y) &= \sum_{k=2}^N \frac{1}{1-\lambda_k} \varphi_k(y)(\varphi_k(y) - \varphi_k(x)).
        \end{align}
    Similarly, the commute time admits the spectral representation
        \begin{align}\label{eq:commute-time-spectral}
            C(x, y) &= \sum_{k=2}^N \frac{1}{1-\lambda_k} (\varphi_k(x) - \varphi_k(y))^2.
        \end{align}
\end{lemma}

The following theorem gives a closed-form spectral representation of the commute time between subgraphs in the setting of the graph edit Markov chain.

\begin{theorem}\label{thm:commute-time-spectral-gemc}
    Let $\mathcal{H} = (\mathcal{V},\mathcal{E})$ be a fixed host graph and let $\mathbf{p}=(p_e)_{e\in\mathcal{E}}$ be fixed with $0<p_e<1$ for each $e\in\mathcal{E}$. Let $(G_t)_{t\geq 0}$ be obtained from the simple edit process with edge probabilities $\mathbf{p} = (p_e)_{e\in\mathcal{E}}$, let $\pi$ denote the corresponding stationary distribution, and let $\{\psi_T\}_{T\subseteq\mathcal{E}}$ denote the eigenvectors of the symmetrized transition probability matrix as in~\cref{thm:orthonormality}. Letting $E, F\subseteq\mathcal{E}$ be fixed, the commute time $C(E, F)$ admits the following spectral representation:
        \begin{align}\label{eq:commute-time-gemc}
            C(E,F)=\sum_{\substack{T\cap (E\triangle F) \neq \varnothing\\ T\neq\mathcal{E}}}\frac{m}{m-|T|}\left(\frac{\psi_T(E)}{\sqrt{\pi(E)}}-\frac{\psi_T(F)}{\sqrt{\pi(F)}}\right)^2,
        \end{align}
    where $E\triangle F$ is the symmetric difference of $E, F$.
\end{theorem}

\begin{proof}
    By~\cref{thm:orthonormality} and~\cref{lem:commute}, it follows that
        \begin{align}\label{eq:spectral-commute-time-1}
            C(E,F)=\sum_{\substack{T\subseteq \mathcal{E}\\ T\neq\mathcal{E}}}\frac{m}{m-|T|}\left(\frac{\psi_T(E)}{\sqrt{\pi(E)}}-\frac{\psi_T(F)}{\sqrt{\pi(F)}}\right)^2.
        \end{align}
    All that remains to be shown is that the sum in~\cref{eq:spectral-commute-time-1} can be taken over only the subsets that have nontrivial intersection with $E\triangle F$. It is enough to show that if $T\cap (E\triangle F) = \varnothing$, then 
        \begin{align}\label{eq:equality}
            \frac{\phi_T(E)}{\pi(E)}=\frac{\phi_T(F)}{{\pi(F)}},
        \end{align}
    where $\phi_T$ is as in~\cref{thm:eigenvectors}. We compute, again using~\cref{thm:eigenvectors}, 
        \begin{align*}
            \frac{\phi_T(E)}{\pi(E)} &= \frac{(-1)^{|\mathcal{E}\setminus(E\cup T)|}\prod_{e\in T\cap E}p_e \prod_{e\in T\setminus E}(1-p_e)}{\prod_{e\in E}p_e \prod_{e\notin E}(1-p_e)}
        \end{align*}
        \begin{align}\label{eq:phi-pi}
            &= {\prod_{e\in E\setminus T}p_e^{-1} \prod_{e\notin (E\cup T)}(p_e-1)^{-1}}\notag\\
            &= \prod_{e\notin T} (p_e^{-1} \mathbf{1}_{\{e\in E\}} + (p_e-1)^{-1} \mathbf{1}_{\{e\notin E\}} )
        \end{align}
    and similarly for $\frac{\phi_T(F)}{\pi(F)}$. Now suppose $e\notin T$ is fixed such that $e\notin E\triangle F$. Then $e\in E\cap F$ or $e\notin E\cup F$. In the former case,
        \begin{align*}
            &p_e^{-1} \mathbf{1}_{\{e\in E\}} + (p_e-1)^{-1} \mathbf{1}_{\{e\notin E\}}\\
            &\qquad= p_e^{-1} \mathbf{1}_{\{e\in F\}} + (p_e-1)^{-1} \mathbf{1}_{\{e\notin F\}}= p_e^{-1},
        \end{align*}
    and similarly in the latter case. Therefore if $T\subseteq (E\triangle F)^c $ each term in the product which appears in~\cref{eq:phi-pi} coincides and~\cref{eq:equality} holds, so that~\cref{eq:commute-time-gemc} is proved.
\end{proof}

We now describe a small example which shows~\cref{thm:eigenvectors} and~\cref{thm:commute-time-spectral-gemc} in action.

\begin{figure}
    \begin{center}
    \input{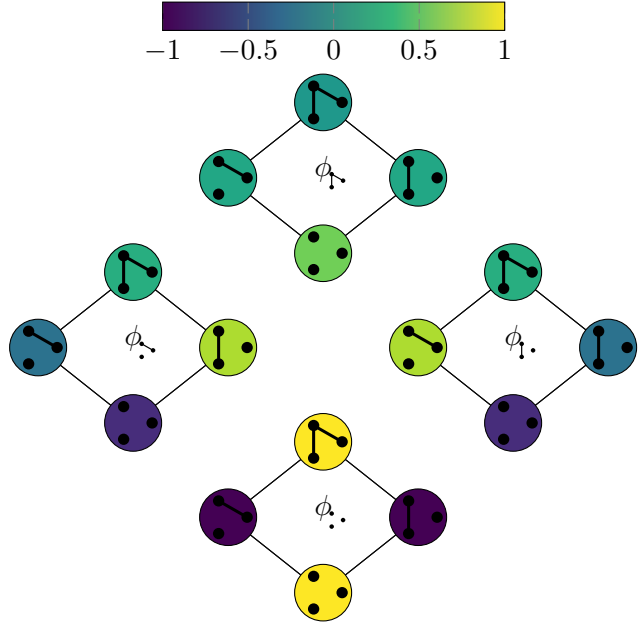}
    \caption{A drawing of the state graph and the corresponding eigenvectors $\phi_T$ of the transition probability matrix for the simple edit process with edge probabilities $\mathbf{p} = (p_e)_{e\in\mathcal{E}}$ on a host graph $\mathcal{H}$ consisting of a path with two edges. We use $p=0.25$ for concreteness. The eigenvectors $\phi_T$ are obtained using the formula in~\cref{thm:eigenvectors}, and each copy of the state graph $\mathscr{S}(\mathcal{H})$ is indexed by a choice of $T$, rendered visually as a subscript of the symbol $\phi$ in each copy.}\label{fig:baby-monster}
    \end{center}
\end{figure}

\begin{example}\normalfont\label{ex:path-on-two-edges}
    In this example we obtain the eigenvectors of the transition probability matrix $\mathsf{P}$ associated to a simple edit process on a small host graph. Let $\mathcal{H} = (\{1, 2, 3\}, \{\{1, 2\}, \{2, 3\}\})$ consist of a path graph on three nodes and two edges, illustrated below. We symbolically denote the two edges $a, b$, respectively, for brevity in the discussion to follow.

    \begin{center}
        \begin{tikzpicture}[scale=0.5]
            \node at (-2, -2)[circle,fill,color=black,inner sep=1.8pt]{};
            \node at (0, 0)[circle,fill,color=black,inner sep=1.8pt]{};
            \node at (2, -2)[circle,fill,color=black,inner sep=1.8pt]{};
            \draw[thick] (-2, -2 )-- (0, 0);
            \draw[thick] (0, 0) -- (2, -2);
            \node at (-1.25, 0) {$a$};
            \node at (1.25, 0) {$b$};
            \node at (-2, -2.5) {$1$};
            \node at (0, -0.5) {$2$};
            \node at (2, -2.5) {$3$};
        \end{tikzpicture}	
    \end{center}

    The edit semigroup is generated by the set $\{{a^+}, {a^-}, {b^+}, {b^-}\}$, and we organize in the table below based on their reduced edit enumerations.

    \begin{center}
        \begin{tabular}{c|c}
            $\supp{x} $ & $x$ \\\hline\hline
            $\varnothing$ & $\widehat{0}$\\
            $\{a\}$ & ${a^+}, {a^-}$\\
            $\{b\}$ & ${b^+}, {b^-}$\\
            $\{a, b\}$ & ${a^+}{b^+}, {a^+}{b^-}, {a^-}{b^+}, {a^-}{b^-}$
        \end{tabular}
    \end{center}

    The corresponding lattice is the Boolean algebra on $2^{\{a, b\}}$,	which has exactly four elements (themselves shown in the leftmost column of the table above). From~\cref{lem:chambers-of-s} we know that the chambers ${\mathcal{C}}$ of $\mathcal{S}$ are exactly the elements which have support of maximal cardinality, which in this case is two. Fixing $p\in(0, 1)$, the transition probability matrix $\mathsf{P}$ for the random walk on the chambers of $\mathcal{S}$ is given by:
        \begin{align*}
            \mathsf{P} &= \begin{bmatrix}
                p & \frac{1-p}{2} & \frac{1-p}{2} & 0\\
                \frac{p}{2} & \frac{1}{2} & 0 & \frac{1-p}{2} \\
                \frac{p}{2}& 0 & \frac{1}{2} & \frac{1-p}{2} \\
                0 &\frac{p}{2} & \frac{p}{2}&1-p \\
            \end{bmatrix}
        \end{align*}
    where the rows and columns are indexed by $\{{a^+}{b^+},{a^+}{b^-}, {a^-}{b^+}, {a^-}{b^-}\}$. By~\cref{lem:ergodicity-of-gemc} we know that the stationary distribution $\pi$ is
        \begin{align*}
            \pi = \begin{bmatrix}
                p^2 & p(1-p) & p(1-p) & (1-p)^2
            \end{bmatrix},
        \end{align*}
    and that $\lambda_{\{a, b\}} = 1$. From~\cref{thm:eigenvectors}, we have
        \begin{align*}
            \phi_{\{b\}} &= \begin{bmatrix}
                p & (1-p) & -p & -(1-p)
            \end{bmatrix},\hspace{.5cm}\lambda_{\{b\}} = \frac{1}{2},\\
            \phi_{\{a\}} &= \begin{bmatrix}
                p & -p & (1-p) & -(1-p)
            \end{bmatrix},\hspace{.5cm}\lambda_{\{a\}} = \frac{1}{2},\\
            \phi_{\varnothing} &= \begin{bmatrix}
                1 & -1 & -1 & 1
            \end{bmatrix},\hspace{.5cm}\lambda_{\varnothing} = 0.\\
        \end{align*}
    Thus we have the diagonalization $\mathsf{P} = U^{-1} \Lambda U$ of $P$, where
        \begin{align}\label{eq:u}
            U = \begin{bmatrix}
                p^2 & p(1-p) & p(1-p) & (1-p)^2 \\
                p & (1-p) & -p & -(1-p) \\ 
                p & -p & (1-p) & -(1-p) \\ 
                1 & -1 & -1 & 1
            \end{bmatrix},
        \end{align}
    and $\Lambda$ is the diagonal matrix of eigenvalues given by
        \begin{align}\label{eq:lambda}
            \Lambda = \begin{bmatrix}
                1 & & & \\
                & \frac{1}{2} & &\\
                & & \frac{1}{2} & \\
                & & & 0
            \end{bmatrix}.
        \end{align}
    Finally, by applying~\cref{thm:commute-time-spectral-gemc} together with the eigenvalues and eigenvectors from~\cref{eq:u} and~\cref{eq:lambda}, the matrix of commute times $C\in\mathbb{R}^{4\times 4}$ with entries $C(E, F)$ for each $E, F\subseteq\mathcal{E}$ is given by
        {\begin{align*}
            C =
                \begin{bmatrix}
                0 & \dfrac{1+p}{p^{2}(1-p)} & \dfrac{1+p}{p^{2}(1-p)} & \dfrac{1}{p^{2}(1-p)^{2}}\\[8pt]
                \dfrac{1+p}{p^{2}(1-p)} & 0 & \dfrac{4}{p(1-p)} & \dfrac{2-p}{p(1-p)^{2}}\\[8pt]
                \dfrac{1+p}{p^{2}(1-p)} & \dfrac{4}{p(1-p)} & 0 & \dfrac{2-p}{p(1-p)^{2}}\\[8pt]
                \dfrac{1}{p^{2}(1-p)^{2}} & \dfrac{2-p}{p(1-p)^{2}} & \dfrac{2-p}{p(1-p)^{2}} & 0
                \end{bmatrix}.
        \end{align*}}
\end{example}

\section{Discussion}\label{sec:discussion}

In this paper we have investigated several edit-based graph evolving processes which satisfy the memoryless conditions. By using methods from semigroup spectral theory, we give a detailed spectral analysis of the transition probability matrices associated to both simple and compound edit processes. We anticipate that these tools can be useful for problems arising from a wide range of evolving processes. Further questions can be asked about extensions and modifications, as well as comparisons of memoryless processes with others.


\begin{thebibliography}{10}

\bibitem{aggarwalEvolutionaryNetworkAnalysis2014}
{\sc C.~Aggarwal and K.~Subbian}, {\em Evolutionary {{Network Analysis}}: {{A Survey}}}, ACM Computing Surveys, 47 (2014), pp.~1--36.

\bibitem{aiello2001random}
{\sc W.~Aiello, F.~Chung, and L.~Lu}, {\em A random graph model for power law graphs}, Experimental mathematics, 10 (2001), pp.~53--66.

\bibitem{akrida2020fast}
{\sc E.~C. Akrida, G.~B. Mertzios, S.~Nikoletseas, C.~Raptopoulos, P.~G. Spirakis, and V.~Zamaraev}, {\em How fast can we reach a target vertex in stochastic temporal graphs?}, Journal of Computer and System Sciences, 114 (2020), pp.~65--83.

\bibitem{anari2019log}
{\sc N.~Anari, K.~Liu, S.~O. Gharan, and C.~Vinzant}, {\em Log-concave polynomials ii: high-dimensional walks and an fpras for counting bases of a matroid}, in Proceedings of the 51st Annual ACM SIGACT Symposium on Theory of Computing, 2019, pp.~1--12.

\bibitem{bencs2023upper}
{\sc F.~Bencs and P.~Csikv{\'a}ri}, {\em Upper bound for the number of spanning forests of regular graphs}, European Journal of Combinatorics, 110 (2023), p.~103677.

\bibitem{bidigare1999combinatorial}
{\sc P.~Bidigare, P.~Hanlon, and D.~Rockmore}, {\em {A combinatorial description of the spectrum for the Tsetlin library and its generalization to hyperplane arrangements}}, Duke Mathematical Journal, 99 (1999), pp.~135 -- 174.

\bibitem{bienvenuMoranForest2021}
{\sc F.~Bienvenu, J.-J. Duchamps, and F.~{Foutel-Rodier}}, {\em The {{Moran}} forest}, Random Structures \& Algorithms, 59 (2021), pp.~155--188.

\bibitem{brownSemigroupsRingsMarkov2000}
{\sc K.~S. Brown}, {\em Semigroups, rings, and markov chains}, Journal of Theoretical Probability, 13 (2000), pp.~871--938.

\bibitem{brown1998random}
{\sc K.~S. Brown and P.~Diaconis}, {\em Random walks and hyperplane arrangements}, Annals of Probability,  (1998), pp.~1813--1854.

\bibitem{butlerEdgeFlippingComplete2015}
{\sc S.~Butler, F.~Chung, J.~Cummings, and R.~Graham}, {\em Edge flipping in the complete graph}, Advances in Applied Mathematics, 69 (2015), pp.~46--64.

\bibitem{chungEdgeFlippingGraphs2012}
{\sc F.~Chung and R.~Graham}, {\em Edge flipping in graphs}, Advances in Applied Mathematics, 48 (2012), pp.~37--63.

\bibitem{chung2012hypergraph}
{\sc F.~Chung and A.~Tsiatas}, {\em Hypergraph coloring games and voter models}, in Algorithms and Models for the Web Graph: 9th International Workshop, WAW 2012, Halifax, NS, Canada, June 22-23, 2012. Proceedings 9, Springer, 2012, pp.~1--16.

\bibitem{clementi2008flooding}
{\sc A.~E. Clementi, C.~Macci, A.~Monti, F.~Pasquale, and R.~Silvestri}, {\em Flooding time in edge-markovian dynamic graphs}, in Proceedings of the twenty-seventh ACM symposium on Principles of distributed computing, 2008, pp.~213--222.

\bibitem{cooper2007sampling}
{\sc C.~Cooper, M.~Dyer, and C.~Greenhill}, {\em Sampling regular graphs and a peer-to-peer network}, Combinatorics, Probability and Computing, 16 (2007), pp.~557--593.

\bibitem{demirci2023mixing}
{\sc Y.~E. Demirci, {\"U}.~I{\c{s}}lak, and A.~{\"O}zdemir}, {\em Mixing time bounds for edge flipping on regular graphs}, Journal of Applied Probability, 60 (2023), pp.~1317--1332.

\bibitem{durrettProbabilityModelsDNA2002}
{\sc R.~Durrett}, {\em Probability {{Models}} for {{DNA Sequence Evolution}}}, Probability and Its {{Applications}}, Springer New York, 2002.

\bibitem{durrett2008probability}
{\sc R.~Durrett}, {\em Probability models for DNA sequence evolution}, vol.~2, Springer, 2008.

\bibitem{godehardt2001two}
{\sc E.~Godehardt and J.~Jaworski}, {\em Two models of random intersection graphs and their applications}, Electronic Notes in Discrete Mathematics, 10 (2001), pp.~129--132.

\bibitem{holmeTemporalNetworks2012}
{\sc P.~Holme and J.~Saram{\"a}ki}, {\em Temporal networks}, Physics Reports, 519 (2012), pp.~97--125.

\bibitem{levin2017markov}
{\sc D.~A. Levin and Y.~Peres}, {\em Markov chains and mixing times}, vol.~107, American Mathematical Soc., 2017.

\bibitem{michailIntroductionTemporalGraphs2016}
{\sc O.~Michail}, {\em An {{Introduction}} to {{Temporal Graphs}}: {{An Algorithmic Perspective}}{\textsuperscript{*}}}, Internet Mathematics, 12 (2016), pp.~239--280.

\bibitem{milewska2025dynamic}
{\sc M.~Milewska, R.~van~der Hofstad, and B.~Zwart}, {\em Dynamic random intersection graph: Dynamic local convergence and giant structure}, Random Structures \& Algorithms, 66 (2025), p.~e21264.

\bibitem{pareEpidemicProcessesTimeVarying2018}
{\sc P.~E. Pare, C.~L. Beck, and A.~Nedic}, {\em Epidemic {{Processes Over Time-Varying Networks}}}, IEEE Transactions on Control of Network Systems, 5 (2018), pp.~1322--1334.

\bibitem{przytyckaDynamicInteractomeIts2010}
{\sc T.~M. Przytycka, M.~Singh, and D.~K. Slonim}, {\em Toward the dynamic interactome: It's about time}, Briefings in Bioinformatics, 11 (2010), pp.~15--29.

\bibitem{saliola2012eigenvectors}
{\sc F.~Saliola}, {\em Eigenvectors for a random walk on a left-regular band}, Advances in Applied Mathematics, 48 (2012), pp.~306--311.

\bibitem{sekaraFundamentalStructuresDynamic2016}
{\sc V.~Sekara, A.~Stopczynski, and S.~Lehmann}, {\em Fundamental structures of dynamic social networks}, Proceedings of the National Academy of Sciences, 113 (2016), pp.~9977--9982.

\bibitem{singer1996random}
{\sc K.~B. Singer}, {\em Random intersection graphs}, The Johns Hopkins University, 1996.

\bibitem{skardingFoundationsModelingDynamic2021a}
{\sc J.~Skarding, B.~Gabrys, and K.~Musial}, {\em Foundations and {{Modeling}} of {{Dynamic Networks Using Dynamic Graph Neural Networks}}: {{A Survey}}}, IEEE Access, 9 (2021), pp.~79143--79168.

\bibitem{spadon2021pay}
{\sc G.~Spadon, S.~Hong, B.~Brandoli, S.~Matwin, J.~F. Rodrigues-Jr, and J.~Sun}, {\em Pay attention to evolution: Time series forecasting with deep graph-evolution learning}, IEEE Transactions on Pattern Analysis and Machine Intelligence, 44 (2021), pp.~5368--5384.

\bibitem{stanley2011enumerative}
{\sc R.~P. Stanley}, {\em Enumerative combinatorics volume 1 second edition}, Cambridge studies in advanced mathematics,  (2011).

\bibitem{wang2018neural}
{\sc T.~Wang, Y.~Zhou, S.~Fidler, and J.~Ba}, {\em Neural graph evolution: Automatic robot design}, in International Conference on Learning Representations, 2019.

\bibitem{wu2020evonet}
{\sc C.~Wu, G.~Nikolentzos, and M.~Vazirgiannis}, {\em Evonet: A neural network for predicting the evolution of dynamic graphs}, in Artificial Neural Networks and Machine Learning--ICANN 2020: 29th International Conference on Artificial Neural Networks, Bratislava, Slovakia, September 15--18, 2020, Proceedings, Part I 29, Springer, 2020, pp.~594--606.

\bibitem{xu2022novel}
{\sc M.~Xu, P.~Fu, B.~Liu, H.~Yin, and J.~Li}, {\em A novel dynamic graph evolution network for salient object detection}, Applied intelligence, 52 (2022), pp.~2854--2871.

\end{thebibliography}
\end{document}